\documentclass[12pt]{amsart}

\usepackage{mathrsfs}
\usepackage{amssymb}
\usepackage{latexsym}
\usepackage{amsmath}
\usepackage[mathscr]{eucal}

\usepackage[all]{xy}
\usepackage{pb-diagram,pb-xy}
\usepackage{pstricks}

\usepackage[a4paper, hmargin=3.0cm, tmargin=3.2cm, bmargin=3.2cm]{geometry}
\numberwithin{equation}{section}

\theoremstyle{plain}
\newtheorem{theorem}[equation]{Theorem}
\newtheorem{lemma}[equation]{Lemma}
\newtheorem{proposition}[equation]{Proposition}
\newtheorem{corollary}[equation]{Corollary}

\theoremstyle{definition}
\newtheorem{definition}[equation]{Definition}
\newtheorem{remark}[equation]{Remark}

\newcommand{\bC}{\mathbb{C}}

\newcommand{\bF}{\mathbb{F}}

\newcommand{\bN}{\mathbb{N}}

\newcommand{\bP}{\mathbb{P}}
\newcommand{\bQ}{\mathbb{Q}}
\newcommand{\bR}{\mathbb{R}}

\newcommand{\bZ}{\mathbb{Z}}

\newcommand{\sD}{\mathscr{D}}

\newcommand{\sF}{\mathscr{F}}

\newcommand{\cO}{\mathcal{O}}

\newcommand{\cR}{\mathcal{R}}

\newcommand{\fE}{\mathfrak{E}}

\newcommand{\fH}{\mathfrak{H}}

\newcommand{\fK}{\mathfrak{K}}

\newcommand{\fM}{\mathfrak{M}}

\newcommand{\fP}{\mathfrak{P}}

\newcommand{\fT}{\mathfrak{T}}
\newcommand{\fU}{\mathfrak{U}}

\newcommand{\fW}{\mathfrak{W}}
\newcommand{\fX}{\mathfrak{X}}
\newcommand{\fY}{\mathfrak{Y}}

\newcommand{\Aut}{\operatorname{Aut}}
\newcommand{\car}{\operatorname{char}}

\newcommand{\Ho}{\operatorname{Ho}}
\newcommand{\id}{\operatorname{id}}
\newcommand{\inc}{\operatorname{inc}}
\newcommand{\Q}{Q}
\newcommand{\DL}{\operatorname{DL}}
\newcommand{\loopinf}{\Omega^{\infty}}
\newcommand{\suspinf}{\Sigma^{\infty}}

\newcommand{\leg}{\operatorname{leg}}

\newcommand{\ver}{\operatorname{Vert}}

\newcommand{\colim}{\operatorname{colim}}

\newcommand{\mnbar}[2]{\overline{\fM}_{#1,#2}}
\newcommand{\M}{\overline{\fM}}

\newcommand{\trf}{\operatorname{trf}}

\newcommand{\thom}{\operatorname{th}}
\newcommand{\proj}{\operatorname{proj}}
\newcommand{\groupcompletion}{\operatorname{gc}}
\newcommand{\dash}{^{\prime}}
\newcommand{\hq}{/\hspace{-1.2mm}/}

\newcommand{\PT}{\operatorname{PT}}
\newcommand{\Sing}{\operatorname{Sing}}
\newcommand{\diff}{\mbox{\textbf{diff}}}
\newcommand{\sch}{\mbox{\textbf{sch}}}
\newcommand{\ttop}{\mbox{\textbf{top}}}
\newcommand{\bTh}{\mathbb{T}\mathbf{h}}
\newcommand{\algstacks}{\textsc{Stacks}^{\sch}}
\newcommand{\topstacks}{\textsc{Stacks}^{\ttop}}

\newcommand{\supp}{\operatorname{supp}}
\newcommand{\End}{\operatorname{End}}
\newcommand{\Mon}{\operatorname{Mon}}
\newcommand{\Hom}{\operatorname{Hom}}

\begin{document}

\title[Pontrjagin-Thom maps and the moduli stack of stable curves]{Pontrjagin-Thom maps and the homology of the moduli stack of stable curves}

\author{Johannes Ebert  \and Jeffrey Giansiracusa}

\maketitle

\begin{abstract}
  We study the singular homology (with field coefficients) of the
  moduli stack $\mnbar{g}{n}$ of stable $n$-pointed complex curves of
  genus $g$.  Each irreducible boundary component of $\M_{g,n}$
  determines via the Pontrjagin-Thom construction a map from
  $\mnbar{g}{n}$ to a certain infinite loop space whose homology is
  well understood.  We show that these maps are surjective on homology
  in a range of degrees proportional to the genus.  This detects many
  new torsion classes in the homology of $\mnbar{g}{n}$.
\end{abstract}

\setcounter{tocdepth}{1}
\tableofcontents

\section{Introduction}

Let $\mnbar{g}{n}$ denote the moduli stack of stable nodal complex
curves of genus $g$ with $n$ labeled marked points; this is the
Deligne-Mumford-Knudsen compactification of the moduli stack
$\fM_{g,n}$ of smooth curves.  This object plays a central role in
Gromov-Witten theory, conformal field theory, and conjecturally in
string topology.  The rational cohomology of $\mnbar{g}{n}$ and its
tautological subalgebra have been extensively studied in the
literature, and the structure of the tautological algebra is at least
conjecturally known.  However, the mod $p$ (co)homology has received
relatively little attention.  Here the distinction between the
\emph{moduli stack} and the associated \emph{coarse moduli space}
becomes important because they are only \emph{rationally} homology
isomorphic.  We take the point of view that the moduli stack is the
more fundamental object.

Using the proof of the integrally refined Mumford conjecture by Madsen
and Weiss \cite{Madsen-Weiss}, Galatius \cite{Galatius} completely
computed the mod $p$ homology of $\fM_{g,n}$ in the Harer-Ivanov
stable range; there are large families of torsion classes.  Here we
address the question of torsion in the homology of the compactified
moduli stack.

The boundary $\partial \mnbar{g}{n}:= \mnbar{g}{n} \smallsetminus
\fM_{g,n}$ of $\mnbar{g}{n}$ is a union of substacks of complex
codimension $1$.  These irreducible boundary components are the images
of the `gluing' morphisms between moduli stacks
defined by identifying two marked points together to form a node.  Let $P$
be a subset of $\{1, \ldots, n\}$.  The gluing morphisms are:
\begin{equation}\label{naturalmorphisms}
\begin{split}
\xi_{irr}:\: & \mnbar{g-1}{n+(2)} \to \mnbar{g}{n}, \\
\xi_{h,P}:\: & \mnbar{h}{P\sqcup \{p_1\}} \times \mnbar{g-h}{P^c\sqcup
 \{p_2\}}\to \mnbar{g}{n},
\end{split}
\end{equation}
where $\mnbar{g-1}{n+(2)}$ is the moduli stack of stable curves with
$n+2$ marked points, the first $n$ of which are labeled. These
morphisms are representable proper immersions (in fact embeddings when
$P$ is a proper subset) of complex codimension $1$ with transversal
(self-)intersections and their images are precisely the various
irreducible components of the boundary.

We study the effect on homology of the
Pontrjagin-Thom maps for these immersions.  We show that the
(self)-intersections produce large families of torsion homology
classes which are unrelated to the known torsion classes on
$\fM_{g,n}$.

Recall that if $N^{n-k}\looparrowright M^n$ is a proper immersion of
real codimension $k$ between smooth manifolds then the classical
Pontrjagin-Thom construction produces a map $M \to \Q N^{\nu(f)}$,
where $\Q X=\Omega^\infty \Sigma^\infty X$ is the free infinite loop
space generated by $X$, and $N^{\nu(f)}$ is the Thom space of the
normal bundle $\nu(f)$.  A reduction of the structural group of
$\nu(f)$ to $G \stackrel{j}{\to} GL_k(\bR)$ induces a map
\[
\Q N^{\nu(f)} \to \Q BG^{j^*\gamma_k},
\]
where $\gamma_k$ is the universal $k$-plane bundle over $BGL_k(\bR)$.
Thus we obtain a map
\[
M\to \Q BG^{j^* \gamma_k}.
\]
In section \ref{ptconstructionsection} we extend the classical construction of
Pontrjagin-Thom maps to the category of differentiable local quotient
stacks.  A stack $\fX$ admitting an atlas has an associated
\emph{homotopy type} $\Ho(\fX)$ (see section \ref{hostacks}) which is
a space that has the same homological invariants as the stack, and the
Pontrjagin-Thom construction produces a map out of the homotopy type.

Let $T(2)=U(1) \times U(1)$ denote the maximal torus in $U(2)$, and
let $N(2) \cong U(1) \wr \bZ/2 = U(1)^2 \rtimes \bZ/2$ denote the
normalizer of the maximal torus.  There are homomorphisms
\[
T(2) \hookrightarrow N(2) \to U(1),
\]
where the first arrow is the inclusion and the second is defined by
multiplying the $U(1)$ components together; we write $V$ for the
universal line bundle over $BU(1)$ or its pullback to $BN(2)$ or
$BT(2)$.  The normal bundle of $\xi_{irr}$ comes with a reduction
of structure group to $N(2)$, and the structure group of the normal
bundle of $\xi_{h,P}$ reduces to $T(2)$.  Thus we have Pontrjagin-Thom
maps
\begin{align*}
  \Phi_{irr}:\: & \Ho(\mnbar{g}{n}) \to \Q BN(2)^V, \\
  \Phi^{1}_{h,P}:\: & \Ho(\mnbar{g}{n}) \to \Q BT(2)^V , \\
  \Phi^{0}_{h,P}:\: & \Ho(\mnbar{g}{n}) \to \Q BT(2)^V \to \Q BU(1)^V,
\end{align*}
where $\Phi^{0}_{h,P}$ is the composition of $\Phi^{1}_{h,P}$ with the
map induced by the multiplication $T(2) \to U(1)$. Our main theorem is
the following.
\begin{theorem}\label{mainthmspecial}
Let $g$ and $n$ be fixed. Let $\bF$ be a field (of arbitrary characteristic).
\begin{enumerate}
\item The map $\Phi_{irr}$ is surjective on $H_i(-,\bF)$ for $i \leq (g-2)/4$.
\item The map $\Phi_{h,\emptyset}^1$ is surjective on $H_i(-,\bF)$ for
  $i \leq (h/2-1)$, $i \leq (g-2)/(2h+2)$.
\item The map $\Phi_{h,\emptyset}^0$ is surjective on $H_i(-,\bF)$ for
  $i \leq (g-2)/(2h+2)$.
\end{enumerate}
\end{theorem}

This theorem detects large families of new torsion classes in the
(co)\-homology of $\mnbar{g}{n}$ as follows.  Let $\Phi$ be one of the
above maps.  On cohomology with field coefficients the induced map
$\Phi^*$ is injective in one of the above ranges of degrees. 
\subsubsection*{Rationally} 
The cohomology of $\Q BN(2)^V$ with coefficients in $\bQ$ is the free
commutative algebra on generators $a_{i,j}$ ($i,j \geq 0$) of degree
$2+2i+4j$.  In this case the image of $\Phi^*$ is contained in the
tautological algebra; see section \ref{comparisontautological}.
\subsubsection*{Mod $p$}
The mod $p$ Betti numbers of $\Q BN(2)^V$ are much larger than the
rational Betti numbers.  If $\car (\bF) > 0$, then $H_*(\Q BN(2)^V;
\bF)$ has a large and rich structure --- it is the free
graded-commutative algebra over the free Dyer-Lashof module generated by
$\widetilde{H}_*(BN(2)^V; \bF)$; see sections \ref{homthomspace} and
\ref{infinite-loop-space-homology} for details.  Hence this detects
large families of new mod $p$ cohomology classes of $\mnbar{g}{n}$
which are \emph{not} reductions of rationally nontrivial classes.

\begin{remark}
\begin{enumerate}
\item More generally, one can take the cartesian product of several of
  these Pontrjagin-Thom maps and the induced map on homology will be
  surjective in a range of degrees.  However, stating the exact range
  of degrees becomes somewhat cumbersome. The more general result is
  Theorem \ref{mainthm}.
\item Note that the range of surjectivity is proportional to $g$ in
  (1) and (3) but not in (2). On the other hand, the homology groups
  of the target in (2) are somewhat larger than those of the target in
  (3), so $\Phi^1_{h,\emptyset}$ detects more classes than
  $\Phi^0_{h,\emptyset}$ but in a reduced range of degrees.
\item When $\emptyset \neq P\subsetneq \{1, \ldots, n\}$ the morphism
  $\xi_{h,P}$ is an embedding.  Therefore its Pontrjagin-Thom map
  factors through $BT(2)^V \to \Q BT(2)^V$. The cohomology classes
  pulled back from $BT(2)^V$ all lie in the tautological ring of
  $\mnbar{g}{n}$.  However, one can also consider the quotient
  $\mnbar{g}{n}\hq \Sigma_n$, where the symmetric group acts by
  permuting the labels of the marked points.  Now the gluing morphism
  \[\xi_{h,P}: \mnbar{h}{P\sqcup \{p_1\}}\hq \Sigma_P \times
  \mnbar{g-h}{P^c\sqcup \{p_2\} }\hq \Sigma_{P^c} \to \mnbar{g}{n} \hq \Sigma_n
  \]
  is an immersion with nontrivial self-intersections whenever $h<g/2$.
  In this case one can easily adapt the proof of Theorem
  \ref{mainthmspecial} to show for instance that the associated
  Pontrjagin-Thom map $\Phi_{h,P}^0$ is surjective on homology in
  degrees $i \leq (g-2)/(2h+2)$, provided that $n\geq |P|
  (g-2)/(2h+2)$.
\item Finally we mention that the restriction of the Pontrjagin-Thom
  maps to the moduli stack $\fM_{g,n}$ of smooth curves is
  nullhomotopic, because the images of the natural morphisms
  (\ref{naturalmorphisms}) lie in $\partial \mnbar{g}{n}$. Thus the
  torsion classes we detect are \emph{not} related to the torsion
  classes on $\fM_{g,n}$ which were computed by Galatius
  \cite{Galatius}.
\end{enumerate}
\end{remark}

There is a certain overlap between Theorem \ref{mainthmspecial} and
unpublished work by Eliashberg and Galatius
\cite{Galatius-Eliashberg}.  They announced a determination of the
homotopy type of the moduli stack of stable irreducible curves as the
genus tends to infinity.  Their result should imply our theorem for
the Pontrjagin-Thom map $\Phi_{irr}$.  However, they do not consider
the other boundary strata.

\subsection*{Outline} In section \ref{hostacks} we recall some
material on stacks and explain the notion of the \emph{homotopy type}
of a topological stack.  In section \ref{ptconstructionsection}, we
show how to generalize the Pontrjagin-Thom construction to proper morphisms
of local quotient stacks. Section \ref{moduli-exposition} reviews some
needed facts about the geometry of the moduli stack $\mnbar{g}{n}$.
In section \ref{maintheorem-section} we state our main theorem in full
generality and prove it.  In section \ref{comparisontautological} we
describe how the classes we detect rationally are related to the
tautological algebra.

\subsection*{Acknowledgements}

The first author was supported by a postdoctoral grant from the German
Academic Exchange Service (DAAD). Some preliminary parts of this work
were done during his stay at the Sonderforschungsbereich
``Geometrische Strukturen in der Mathematik'' at the Mathematical
Institute in M\"unster; he thanks Wolfgang L\"uck for his invitation.
The second author thanks the IHES for its hospitality and C.F.
B\"odigheimer for an invitation to Bonn, where this project was begun.
Both authors thank Ulrike Tillmann for helpful comments.

\section{Some homotopy theory for topological stacks}\label{hostacks}

In this section we set up the homotopical framework in which the
Pontrjagin-Thom maps for stacks will reside.

\subsection{Generalities on stacks}

We will assume that the reader is comfortable with the language of
stacks and therefore we will not repeat the basic definitions in
detail. A stack over a site $\mbox{\textbf{S}}$ is a lax sheaf of
groupoids over $\mbox{\textbf{S}}$.  We will consider the sites
$\diff$ and $\ttop$ of smooth manifolds and topological spaces.  The
reader is referred to \cite{Hein} and \cite{Noo1} for readable
introductions to the theory of stacks over the sites $\diff$ and
$\ttop$.

On the site $\diff$ there is a subtlety in the definition of
representable morphisms since one needs transversality for the
pullback of two smooth maps to be a smooth manifold. We propose a
definition which differs slightly from that given in \cite{Hein}.
\begin{definition}
\begin{enumerate}
\item A morphism $f:\fX \to \fY$ of stacks on the site $\diff$ is a
  \emph{representable submersion} if for any manifold $M$ and any
  morphism $M \to \fX$, the fiber product $M \times_{\fY} \fX$ is a
  smooth manifold and the induced map $M \times_{\fY} \fX \to M$ is a
  submersion.
\item A morphism $f:\fX \to \fY$ of stacks over $\mbox{\textbf{S}}$ is
  \emph{representable} if for any representable submersion $f: M \to
  \fY$, the pullback $M \times_{\fY} \fX$ is a smooth manifold and the
  induced map $M \times_{\fY} \fX \to M$ is a smooth map.
\end{enumerate}
\end{definition}
With this definition any smooth map between manifolds is representable
when considered as a morphism of stacks and any morphism from a smooth
manifold to a stack over $\diff$ is representable.  Let $\fX$ be a stack
over $\diff$. An \emph{atlas} is a smooth manifold $X$ together with a
representable submersion $p:X \to \fX$. A stack which admits an atlas
is called a \emph{differentiable stack}.
 
Similarly, we can define topological stacks. We say that a
representable morphism $f:\fX \to \fY$ of stacks over $\ttop$
\emph{has local sections} if for any space $Y$ and any map $Y \to
\fY$, the pullback $\fX \times_{\fY} Y \to Y$ admits local sections
(observe that maps which have local sections are surjective and having
local sections is a property which is invariant under base-change). An
atlas for a stack $\fX$ over $\ttop$ is a space $X$ together with a
representable morphism $p:X \to \fX$ having local sections. A
\emph{topological stack} is a stack $\fX$ over $\ttop$ which admits an
atlas. Our terminology differs from that used by Noohi \cite{Noo1}:
the topological stacks defined above are called ``pretopological
stacks'' in \cite{Noo1} and his ``topological stacks'' satisfy a
stronger condition.

We write $\textsc{Stacks}^{\mbox{\textbf{S}}}$ for the category of
stacks on $\mbox{\textbf{S}}$ which admit an atlas.  Note that
$\topstacks$ contains the category of spaces as a full subcategory.  A
topological (or differentiable, resp.)  stack is said to be a
\emph{Deligne-Mumford} stack if it has an \'etale atlas, i.e. there is
an atlas $p:\fX \to X$ which is a local homeomorphism (local
diffeomorphism, resp.). A differentiable Deligne-Mumford stack is the
same as an \emph{orbifold}.

There is also the category $\algstacks$ of algebraic stacks, studied
in the book \cite{Laumon}.  Moduli stacks of (stable) curves, which
constitute the example of interested to us, are most conveniently
described (and constructed) as algebraic stacks.  There is a functor
$\algstacks \to \topstacks$ which extends the ``complex points
functor'' and is constructed as follows (for details, see \cite{Noo1},
p. 78 f.). An atlas $X\to \fX$ gives rise to a groupoid object in
schemes $X\times_\fX X \rightrightarrows X$, and the moduli stack of
torsors for this groupoid object is canonically equivalent to the
original stack. Taking complex points with the analytic topology gives
a groupoid in topological spaces which determines a topological stack.
The restriction of this functor to smooth stacks in schemes takes
values in differentiable stacks, and its restriction to smooth
Deligne-Mumford algebraic stacks takes values in differentiable
Deligne-Mumford stacks.

\subsection{The homotopy type of a topological stack}

We now introduce the homotopy type of a topological stack. There is a
folklore definition of the homotopy type as the classifying space of
the groupoid associated to an atlas. We present an axiomatic approach
which is equivalent by Proposition \ref{noohistheorem}.  The content
here is a ideological reemphasis of ideas which have been present in
the literature for some time.  The main technical points of the
following exposition are contained in \cite{Noo2}, although the notion
of a homotopy type does not occur explicitly there.

\begin{definition}
  Let $f:\fX \to \fY$ be a representable morphism of topological
  stacks.  Then $f$ is said to be a \emph{universal weak equivalence}
  if for any test map $Y \to \fY$ from a space $Y$, the left vertical
  map in the diagram
\[
\xymatrix{
Y \times_{\fY} \fX \ar[r] \ar[d] & \fX \ar[d]^{f}\\
Y \ar[r] & \fY}
\]
is a weak homotopy equivalence of topological spaces.

A \emph{homotopy type} for a topological stack $\fX$ is a pair
$(\Ho(\fX),\eta)$, where $\Ho(\fX)$ is a CW complex and
$\eta: \Ho(\fX) \to \fX$ is a universal weak equivalence (which is
automatically representable, by \cite{Noo1}, Corollary 7.3).
\end{definition}

\begin{definition} Let $\fX$ be a topological stack and
$Y$ be a topological space.  A \emph{concordance} between elements
$t_0, t_1 \in \fX(Y)$ is an element $t \in \fX(Y \times [0,1])$,
together with isomorphisms $t|_{Y \times \{i\}} \cong t_i$, $i=0,1$.
The category $\fX(Y)$ is skeletally small and concordance is an
equivalence relation on the objects.  The set of concordance classes
of objects is denoted $\fX[Y]$.
\end{definition}

Note that for spaces $X$ and $Y$, there is a natural bijection between
concordance classes $X[Y]$ and homotopy classes $[Y,X]$.

\begin{lemma}[\cite{Noo2}, Corollary 3.8]\label{liftinglemma}
  Let $\eta:\Ho(\fX) \to \fX$ be a homotopy type for $\fX$. Then for
  each CW complex $Y$ and map $g: Y \to \fX$, there exists a map $h:Y
  \to \Ho(\fX)$ and a concordance between $\eta \circ h$ and $g$.
  Moreover, $h$ is unique up to homotopy (which is the same as
  concordance).
\end{lemma}

In particular, there is a natural bijection between the set of
concordance classes $\fX [Y]$ and the set of homotopy classes of maps
$[Y; \Ho(\fX)]$ when $Y$ is a CW complex.

\begin{corollary}
Any two homotopy types of a topological stack are canonically
homotopy equivalent. Moreover, choosing homotopy types defines a functor
from the category of stacks over $\ttop$ which admit a homotopy type
to the homotopy category of spaces. This functor sends
$2$-isomorphic morphisms of stacks to identical homotopy classes of maps.
\end{corollary}

\begin{proof}
  This corollary follows immediately from Lemma \ref{liftinglemma}
  (for the last sentence, note that $2$-isomorphic morphisms are
  concordant).
\end{proof}

We have not yet seen that a topological stack admits a homotopy type;
this is ans\-wered by Theorem \ref{noohistheoremdash} below.

Let $\fX$ be a topological stack with an atlas $X_0 \to \fX$.  This
determines a simplicial space $X_n = X_0 \times_{\fX} \cdots
\times_{\fX} X_0$ ($n+1$ copies) which is in fact the nerve of the
topological groupoid $X_1 = X_0 \times_{\fX} X_0 \rightrightarrows
X_0$. Let $\| X_{\bullet} \|$ be the \emph{thick} realization of the
simplicial space $X_{\bullet}$. The thick realization of a simplicial
space is obtained by forgetting the degeneracies and using only the
boundary maps. In most cases of interest, the thick geometric
realization and the usual geometric realization are homotopy
equivalent, see \cite[p. 308]{Segcat}.

\begin{proposition}[\cite{Noo2}, Theorem 3.11]
\label{noohistheorem}
  If $X_0 \to \fX$ is an atlas of a topological stack with associated
  simplicial space $X_{\bullet}$, then there is a universal weak
  equivalence $\| X_{\bullet} \| \to \fX$.
\end{proposition}

The space $\|X_{\bullet}\|$ is in general not a CW complex. To produce
a homotopy type, we need a small extra argument. The realization of
the singular simplicial set $| Sing_{\bullet} (\| X_{\bullet} \|)|$ is
a CW complex and the evaluation map $| Sing_{\bullet} (\| X_{\bullet}
\|)| \to \| X_{\bullet} \|$ is both a weak homotopy equivalence and a
Serre fibration. Therefore, the composition $| Sing_{\bullet} (\|
X_{\bullet} \|)| \to \fX$ is a homotopy type. This shows:

\begin{theorem}\label{noohistheoremdash}
Any topological stack admits homotopy type.
\end{theorem}

\subsection*{Homotopy types and group actions}
There is a pleasant interaction between the notion of the homotopy
type of a stack and more familiar topological constructions.

Firstly, if $X$ is a CW complex then we can consider $X$ as a
topological stack.  Clearly, the identity map $X \to X$ is a universal
weak equivalence and thus $\Ho(X) \simeq X$.

An important class of examples of stacks are the \emph{(global) quotient
  stacks}.  Let $G$ be a topological group acting on a space $X$. The
quotient stack $X \hq G$ is defined as follows. If $Y$ is space, then
$X \hq G (Y)$ is the groupoid of triples $(P,p,f)$; $p:P \to Y$ a
principal $G$-bundle and $f:P \to X$ a $G$-equivariant map. The
isomorphisms are defined in the obvious way. There is a natural
morphism $q:X \to X \hq G$ defined as follows: Consider the trivial
principal $G$-bundle $pr_X:G \times X \to X$. Note that $G$ acts on $G
\times X$ only by group multiplication (and not on $X$!) and that the
action map $\mu:G \times X \to X$ is $G$-equivariant. Thus $(G \times
X, pr_X, \mu)$ is an element of $X \hq G (X)$, defining a morphism
$q:X \to X \hq G$. Note that $q$ is a principal $G$-bundle.

\begin{proposition}
The homotopy type of $X \hq G$ is homotopy equivalent to the Borel
construction $EG \times_G X$.
\end{proposition}
\begin{proof}
The projection map $EG \times X \to X$ is $G$-equivariant while the
quotient map $EG \times X \to EG \times_G X$ is a principal
$G$-bundle, so both maps together define a morphism
\[
\eta: EG \times_G X \to X \hq G.
\]
Clearly, $\eta$ is a fiber bundle with structure group $G$ and fiber
$EG$: it is associated to the principal bundle $X \to X \hq G$.
Therefore, if $Y$ is a space and $Y \to X \hq G$ a map, then the
pullback $Y \times_{X \hq G} (EG \times_G X) \to Y$ is a fiber bundle
with contractible fibers, hence a weak homotopy equivalence. Hence
$\eta$ is a universal weak equivalence.
\end{proof}

An important quotient stack is the moduli stack $\fM_{g,n}$ of
smooth complex curves. It is the stack quotient of the Teichm\"uller
space $\fT_{g,n}$ by the action of the mapping class group
$\Gamma_g^n$ of isotopy classes of orientation preserving
diffeomorphism of a genus $g$ surface with $n$ marked points.  Hence
\[
\Ho(\fM_{g,n}) \simeq E\Gamma_g^n \times_{\Gamma_g^n} \fT_{g,n}
             \simeq B\Gamma_g^n,
\]
because the Teichm\"uller space is contractible.

We will have occasion to deal with group actions on stacks.  Suppose
$\fX$ is a topological stack with a strict action of a group $G$ (i.e.
the action is not just up to coherent 2-morphisms). We will not have
to care about group actions which are not strict. Given a strict
$G$-action on $\fX$ and a $G$-space $Y$, the notion of an equivariant
morphism $Y \to \fX$ is well-defined.

There are two equivalent descriptions of principal $G$-bundles over a
stack $\fX$: as a morphism $\fX \to \ast \hq G$, or as a stack $\fP$
with a strict $G$-action and a $G$-invariant representable morphism
$\fP \to \fX$ such that the pullback $\fP \times_{\fX} X \to X$ along
any morphism $X \to \fX$ is a principal $G$-bundle in the usual sense.
An analogous remark applies to arbitrary fiber bundles.

The quotient stack $\fX \hq G$ is defined in the same way as $X \hq
G$ for spaces $X$: for a space $Y$, an object of $(\fX \hq G) (Y)$
consists of a principal $G$-bundle $P \to Y$ and a $G$-equivariant
morphism $P \to \fX$. Again, it is clear that $ \fX \to \fX \hq G$ is a
principal $G$-bundle.

\begin{proposition}
Let $\fX$ be a topological stack with a strict $G$-action. Then the
following hold.
\begin{enumerate}
\item $\fX \hq G$ is a also a topological stack.
\item There exists a homotopy type $\Ho(\fX)$ which is a principal
bundle on $\Ho(\fX \hq G)$ such that the universal morphism $\Ho(\fX )
\to \fX$ is $G$-equivariant.
\item $\Ho(\fX \hq G) \simeq EG \times_G \Ho(\fX)$.
\end{enumerate}
\end{proposition}

\begin{proof} Let $X \to \fX$ be an atlas, i.e. a representable
  morphism which admits local sections. Because $\fX \to \fX\hq G$ is
  a bundle, the composite $X \to \fX \hq G$ is clearly a
  representable morphism with local sections.  This shows (1).

For (2), choose a homotopy type $\Ho(\fX \hq G) \to \fX
\hq G$ and consider the fiber-square 
\[
\xymatrix{ \Ho(\fX \hq G) \times_{\fX \hq G} \fX \ar[r] \ar[d] &
\fX \ar[d]\\
\Ho(\fX \hq G ) \ar[r] & \fX \hq G .}
\]
Because the right vertical map is a principal $G$-bundle, so is the
left vertical map. Because the bottom horizontal map is a universal
weak equivalence, the top horizontal is also a universal weak
equivalence. Thus the space $\Ho(\fX \hq G) \times_{\fX \hq G} \fX \to
\fX$ is a homotopy type for $\fX$ and is also $G$-equivariant, which
shows (2).

For (3), observe that the natural map $EG \times_G
\Ho(\fX) \to \Ho(\fX) /G = \Ho(\fX \hq G)$ is a fiber bundle with
fiber $EG$, hence a weak homotopy equivalence.
\end{proof}

\subsection*{Homology of a topological stack}

The definition of the homotopy type of a stack is justified both by
the above examples and by the fact, which we now explain, that the
space $\Ho(\fX)$ has the correct (co)homology.  A topological stack
has singular (co)homology and sheaf cohomology.  These turn out to be
canonically isomorphic to the (co)homology of the space $\Ho(\fX)$.

The following definition of singular homology for stacks is from
\cite{Behr}.  An atlas $X\ \to \fX$ determines a simplicial space
$X_\bullet$.  Applying $\Sing_\bullet$ produces a bisimplicial set
which generates a double complex $C_{\bullet,\bullet}(\fX)$ of Abelian
groups.  The singular homology $H_{*}^{sing}(\fX)$ of $\fX$ is defined
to be the homology of the total complex
$Tot(C_{\bullet,\bullet}(\fX))$. It can be shown that this is
independent of the choice of atlas. There is a map of simplicial spaces 
\[(p \mapsto |\Sing_{\bullet} X_p|) \to (p \mapsto X_p); 
\]
the homology of the (realization of the) left hand side is
$H_{*}^{sing}(\fX)$, the homology of the right-hand side is
$H_{*}(\Ho(\fX))$. A straightforward application of the homology
spectral sequence of a simplicial space \cite{segclass} shows that the
induced map on homology is an isomorphism. Thus we have a natural
isomorphism
\begin{equation}\label{isohomology}
  H_{*}^{sing}(\fX) = H_*(Tot(C_{\bullet,\bullet}(\fX))) \cong H_*(|X_\bullet|) = H_*(\Ho(\fX)).
\end{equation}

The singular cohomology of $\fX$ is defined analogously and it
agrees with the sheaf cohomology by standard arguments.  By the same
reasoning as before, the singular cohomology is canonically isomorphic
to $H^*(\Ho(\fX))$.

For a topological stack $\fX$, let $\fX^{coarse}$ be the coarse moduli
space (this is the orbit space of a groupoid presenting $\fX$).  There
is a natural map $\fX \to \fX^{coarse}$ (which is almost never
representable) and the composition
\begin{equation}\label{universalmap}
\mu_{\fX}:\Ho(\fX) \to \fX \to \fX^{coarse}
\end{equation}
is a rational homology equivalence when $\fX$ is an orbifold (see e.g.
\cite{Haef}).

When $\fX$ is an orbifold it has an orbifold fundamental group
$\pi_1^{orb}\fX$ (see \cite{Moer}), and there is a natural
isomorphism $\pi_1 \Ho(\fX) \cong \pi_1^{orb}\fX$.  One can introduce
coefficient systems, and the isomorphisms (\ref{isohomology}) of
(co)homology hold also for twisted coefficients.

\section{The Pontrjagin-Thom construction for differentiable stacks}\label{ptconstructionsection}

In this section we describe an extension of the classical
Pontrjagin-Thom construction of homotopy-theoretic wrong-way maps to
the setting of differentiable stacks.

\subsection{Preliminaries on stable vector bundles and Thom spectra}

If $W \to X$ is a real vector bundle then the \emph{Thom space} of
$W$, denoted $X^W$ is the space obtained by taking the fiberwise
one-point compactification of $W$ and then collapsing the section at
infinity to the basepoint.  (If $X$ is compact then this is simply the
one-point compactification of $W$.)

A \emph{virtual vector bundle} on a space $X$ is a pair $(E_0, E_1)$
of real vector bundles on $X$; one should think of it as the formal
difference $E_0 - E_1$, and we will sometimes use this more suggestive
notation.  The \emph{rank} of $(E_0, E_1)$ is the difference $\dim E_0
-\dim E_1$.  An isomorphism $(E_0,E_1) \to (F_0,F_1)$ is represented by
a pair $(V,\theta)$ where $V$ is a vector bundle and 
\[
\theta: E_0 \oplus F_1 \oplus V \to E_1\oplus F_0 \oplus V
\]
is a bundle isomorphism.  Two pairs $(\theta,V)$, $(\theta \dash,
V\dash)$ represent the same morphism if there exists a vector bundle
$U$ such that $V \dash = V \oplus U$ and $\theta \dash = \theta \oplus
\id_U$ (and then take the equivalence relation that this generates).
The composition of $\theta: E_0 \oplus F_1 \oplus V \to E_1 \oplus F_0
\oplus V$ and $\phi: F_0 \oplus G_1 \oplus W \to F_1 \oplus G_0 \oplus
W$ is defined to be $F_1 \oplus V \oplus W$ together the composition
\begin{align*}
E_0 \oplus F_1 \oplus G_1 \oplus V \oplus W 
  \stackrel{\theta \oplus \mathrm{id}_{G_1 \oplus W}}{\longrightarrow} 
E_1 \oplus F_0 \oplus G_1 \oplus V \oplus W \\
   \stackrel{\phi \oplus \mathrm{id}_{E_1 \oplus V}}{\longrightarrow} 
E_1 \oplus F_1 \oplus G_0 \oplus V \oplus W.
\end{align*}
The category of virtual vector bundles over a
fixed space is a groupoid; these form a presheaf of groupoids on the
site $\ttop$.  Let $\fK$ denote the stackification of the above
presheaf. The objects of this stack are slightly more general than
virtual bundles; they can locally be presented as formal differences
of vector bundles, but globally this might be impossible.  Objects of
$\fK$ are called \emph{stable vector bundles}.

Let $\fK_d$ denote the full substack consisting of virtual bundles of rank
$d$.  For $n \geq d$, let $\ast \to \fK_d$ be the arrow representing
the stable vector bundle $(\bR^n;\bR^{n-d})$.  It is easy to see that
this is an atlas for $\fK_d$ (as a topological stack) and in fact
$\fK_d$ is equivalent to the stack $* \hq O$.  Thus 
\[
\Ho(\fK) = \coprod_{d \in \bZ} \Ho(\fK_d) \simeq \bZ \times BO,
\]
as expected, and $2$-isomorphism classes of morphisms $X \to \fK$
correspond to homotopy classes $X \to \bZ \times BO$.

For any map $c_W:X \to \{d \}\times BO$ which classifies a stable
vector bundle $W$ of rank $d$, there is an associated \emph{Thom
  spectrum} $\bTh(W)$, produced as follows.  There is an exhaustive
filtration $X_{-d} \subset X_{1-d} \subset \cdots \subset X$, where
$X_n := c_{W}^{-1}(\{d \} \times BO_{d+n})$.  Let $W_n := c_{W}^{*}
\gamma_{d+n}$ be the pullback of the $d+n$-dimensional universal
vector bundle. Clearly, there is an isomorphism $W_{n+1} |_{X_n} \cong
\bR \oplus W_n$.  The $n^{th}$ space of $\bTh(W)$ is $X_n^{W_n}$ and
the structure maps are
\[
\Sigma X_n^{W_n} \cong X_n^{\bR \oplus W_n } \cong X_n^{W_{n+1}|_{X_n}}
\hookrightarrow X_{n+1}^{W_{n+1}}.
\]
The homotopy type of the spectrum $\bTh(W)$ depends only on the
homotopy class of $c_{W}$, which can be viewed as an element in the
real $K$-theory group $ KO^0 (X)$.  Furthermore, when $W$ is
representable by an actual vector bundle $W_0$ then the Thom spectrum
is homotopy equivalent to the suspension spectrum $\suspinf X^{W_0}$
of the Thom space.  The reader who wants to know more details about
Thom spectra is advised to consult \cite{Rudyak}, chapter IV, \S 5.

\subsection{The classical Pontrjagin-Thom construction}

We briefly recall the classical construction.  Let $f:M \to N$ be a
proper smooth map between smooth manifolds of codimension $d$ (i.e.,
$\dim N - \dim M = d$).  The normal bundle
\[
\nu(f) := f^* TN - TM
\] 
is a virtual vector bundle of dimension $d$ on $M$.  For $n$ large
enough there exists an embedding $j:M \hookrightarrow \bR^n \times N$ such that
$pr_N \circ j = f$.  The virtual bundle $\nu(j) - \bR^n$ is
canonically isomorphic to $\nu(f)$.

Choose a tubular neighborhood $U$ of $j(M)$, identify $U\cong \nu(j)$,
and define a map $\bR^n \times N \to M^{\nu(j)}$ as follows: if a
point lies in $U$ then it is mapped to the corresponding point in
$\nu(j) \subset M^{\nu(j)}$; all points outside $U$ are mapped to the
basepoint of $M^{\nu(j)}$.  Because $f$ is proper, this construction
extends to a map $\Sigma^n N_+ \to M^{\nu(j)}$.  The space $M^{\nu(j)}$
is the $n^{th}$ space of the spectrum $\bTh(\nu(f))$ and letting $n$
tend to infinity defines a map of spectra
\[
\PT_f:\suspinf N_+ \to \bTh(\nu(f))
\]
which is the classical Pontrjagin-Thom map.  Recall that the functor
$\suspinf$ from spaces to spectra is left adjoint to the functor
$\loopinf$.  The adjoint map of $\PT_f$ is a map of spaces $N \to
\loopinf \bTh(\nu(f))$, which we also denote by $\PT_f$, because there
is no risk of confusion.  The homotopy class of $\PT_f$ does not depend
on the choices involved.  The Pontrjagin-Thom map can be used to define
umkehr maps in cohomology, see section \ref{comparisontautological}.

\subsection{Normal bundles for stacks and statement of the theorem}

To extend the Pontrjagin-Thom construction to stacks one must be able
to define the normal bundle of a morphism.

Let $f:\fX \to \fY$ be proper representable morphism of differentiable
stacks. The \emph{codimension} $d$ of $f$ is by definition $d=\dim(Y)
- \dim(Y \times_{\fY} \fX)$, where $Y \to \fY$ is an atlas. Let $Y \to
\fY$ be an atlas and let $X :=\fX \times_{\fY} Y \to \fX$ be the
induced atlas for $\fX$. The map $f$ pulls back to a map $f_Y:X \to Y$
which is a proper smooth map.  The normal bundle $f^* TY - TX$ is a
virtual vector bundle on $X$, and so it is classified by a morphism $X
\to \fK_d$.  Since normal bundles are natural with respect to pullback
along submersions, this morphism descends to a morphism
\[
\nu(f):\fX \to \fK_d
\]
Taking homotopy types produces a map $\Ho(\fX) \to BO$ which then
yields a Thom spectrum $\bTh (\nu(f))$.

If $N$ is a manifold and $g:N \to \fY$ is a map which is transversal
to $f$.  Then we have a pullback diagram
\[
\xymatrix{N \times_{\fY} \fX \ar[r]^-{h} \ar[d]^{f_N} & \fX \ar[d]^{f}\\
N \ar[r]^{g} & \fY,}
\]
where $f_N$ is a proper map of manifolds.  Thus we have a
Pontrjagin-Thom map $\PT_{f_N}: N \to \loopinf \bTh (\nu(f_N))$.
There is an induced morphism $\nu(f_N) \to \nu(f)$ of stable vector
bundles which covers the map $h$ and induces a map $\loopinf \bTh
(\nu(f_N)) \to \loopinf \bTh (\nu(f))$.  Finally recall that $g$ has a
canonical (up to homotopy) lift $g \dash : Y \to \Ho(\fY)$.

\begin{definition}
  Let $f:\fX \to \fY$ be a proper representable morphism of
  differentiable stacks. A \emph{Pontrjagin-Thom map} for $f$ is a map
  $\PT_f:\Ho(\fY) \to \loopinf \bTh (\nu(f))$ such that the following
  diagram is homotopy-commutative:
\[
\xymatrix{ \loopinf \bTh (\nu(f_N)) \ar[r] &  \loopinf \bTh
(\nu(f))\\
N \ar[u]^{\PT_{f_N}} \ar[r]^{g \dash} & \Ho(\fY) \ar[u]^{\PT_f}}
\]
where $g'$ is a lift of a map $g:N \to \fY$ that is transveral to $f$.
\end{definition}

The following is the main result of this section.

\begin{theorem}\label{pontthom}
  Let $f: \fX \to \fY$ be a proper representable morphism of
  differentiable stacks with $\fY$ a local quotient stack (see
  Definition \ref{deflocstack} below). Then there exists a
  Pontrjagin-Thom map
\[
\PT_f : \Ho(\fY) \to \loopinf \bTh(\nu(f)).
\]
\end{theorem}

The map $\PT_f$ is unique in the sense that it depends on a contractible space of choices.
The main ingredient in the proof of this theorem is a variant of the Whitney
Embedding Theorem for local quotient stacks (Prop \ref{whitney}). One
constructs appropriate embeddings for global quotients using standard
equivariant techniques and then glues these together to obtain
embeddings for local quotient stacks. The construction of
Pontrjagin-Thom maps is then a matter of adapting the classical
construction.

\subsection{Local quotient stacks}

Here we introduce local quotient stacks, which we view as the natural setting for
the Pontrjagin-Thom construction.

\begin{definition}\label{deflocstack}
  A \emph{local quotient stack} is a topological
  stack $\fX$, such that
\begin{enumerate}
\item there exists a paracompact atlas for $\fX$.
\item there exists a countable cover of open substacks $\fX_i \subset
  \fX$ such that $\fX_i \cong X_i \hq G_i$ for some space $X_i$ and
  some compact Lie group $G_i$.
\item The diagonal morphism $\fX \to \fX \times \fX$ is representable
  and proper.
\end{enumerate}
A differentiable stack is a local quotient stack if the spaces $X_i$
are smooth manifolds with smooth $G_i$-actions.
\end{definition}

\begin{lemma}[\cite{FHT}, Lemma A.14]\label{fht}
  If $\fY$ is a local quotient stack and $f: \fX \to \fY$ is a
  representable morphism of topological stacks then $\fX$ is a local
  quotient stack as well.  In particular, every open substack of a
  local quotient stack is a local quotient stack. The analogous
  statements for differentiable local quotient stacks are also true.
\end{lemma}
\begin{proof}
  First suppose that $\fY$ is a global quotient $Y \hq G$. Let $X:=
  \fX \times_{\fY} Y \to \fX$ be the atlas of $\fX$ obtained by
  pulling back the atlas $Y \to \fY$. One easily checks that $X
  \times_{\fX} X \cong G \times X$ and that the two arrows
  $X\times_\fX X = G \times X \rightrightarrows X$ are the projection
  onto $X$ and a group action.  Furthermore, one can check that $f:
  \fX \to \fY$ is represented by a $G$-equivariant map $X \to Y$.  Now
  suppose $\fY$ is a local quotient stack with a covering by global
  quotients $\{\fY_i \cong Y \hq G_i \}$.  The substacks $\fX_i :=
  \fY_i \times_{\fY} \fX$ form an open cover of $\fX$ and by the
  above, $\fX_i \cong X_i \hq G_i$.
\end{proof}

Lemma \ref{fht} indicates that the class of local quotient stack is
large and robust.  Orbifolds are local quotient stacks and so are
global quotient stacks of the form $Y \hq \Gamma$, where $\Gamma$ is a
(possibly noncompact) Lie group which acts properly on $Y$.  A very
general result by Zung \cite{Zung} states that any proper Lie groupoid
represents a local quotient stack.

\begin{lemma}
The coarse moduli space $\fX^{coarse}$ of a differentiable local quotient stack $\fX$
is a paracompact Hausdorff space.
\end{lemma}
\begin{proof}
  Given an atlas $X \to \fX$, one sees that the associated groupoid
  $X\times_{\fX}X \rightrightarrows X$ is proper in the sense of
  \cite{Moer}.  The coarse moduli space is the orbit space of this
  groupoid and hence it is Hausdorff and paracompact.
\end{proof}

As an application of this lemma, we have existence of locally finite
smooth partitions of unity subordinate to any open cover of $\fY$ as
follows.  Any open cover of $\fY$ gives an open cover of
$\fY^{coarse}$.  On $\fY^{coarse}$, we have partitions of unity, which
can then be pulled back via $\fY \to \fY^{coarse}$.

\subsection{Universal vector bundles on local quotient stacks}

In this section, we introduce ``universal vector bundles'' on stacks.
Freed, Hopkins and Teleman \cite{FHT} showed that any local quotient
stack admits a universal \emph{Hilbert} bundle.  For the purpose of
constructing Pontrjagin-Thom maps we instead need universal vector
bundles with fiber $\bR^\infty = \colim \bR^n$.  Here we show that any
local quotient stack has a universal countably-dimensional vector
bundle (its completion will be a universal Hilbert bundle).

Consider $\bR^\infty$ with the colimit topology; this is a complete,
locally convex topology which is not metrizable.  It is a very fine
topology: \emph{any} linear map $\bR^{\infty} \to V$ to an arbitrary
topological vector space $V$ is continuous.

Consider the group $O(\bR^{\infty})$ of isometries of $\bR^{\infty}$
with respect to the standard inner product. On $O(\bR^{\infty})$ we
define the following topology.  The compact-open topology on the
vector space $\End(\bR^{\infty})$ agrees with the topology of
pointwise convergence.  Embed $O(\bR^{\infty}) \hookrightarrow
\End(\bR^{\infty}) \times
\End(\bR^{\infty})$ via $f \mapsto (f, f^{-1})$ and take the induced
subspace topology on $O(\bR^{\infty})$. Finally, replace this topology by its
compactly generated replacement \cite{steen}.

\begin{proposition}\label{functionalanalysis} Let $O(\bR^{\infty})$ be
  endowed with the topology described above.
\begin{enumerate}
\item $O(\bR^{\infty})$ is a topological group.
\item By extension, $O(\bR^{\infty})$ acts by isometries on $\ell^2$
  and this action is continuous\footnote{This action cannot be
    extended continuously to an action of the group of all (say
    bounded) isomorphisms of $\bR^{\infty}$.}.
\item Let $V \subset \bR^{\infty}$ be a finite-dimensional subspace.
  Then the standard inclusion $O(V) \to O(\bR^{\infty})$ is
  continuous.
\item Let $G$ be a compact Lie group and let $V$ be a
  countably-dimensional orthogonal $G$-representation (then $V$ is
  isometric to $\bR^{\infty}$). The action homomorphism $G \to O(V)$
  is continuous.
\end{enumerate}
\end{proposition}
\begin{proof}
  Claim (1) follows immediately from Theorem 5.9 in \cite{steen}.
  For (2), by Theorem 3.2.(i) of \cite{steen}, it suffices to show
  that $O(\bR^{\infty}) \times \ell^2 \to \ell^2$ is continuous when
  $O(\bR^{\infty})$ has the compact-open topology, which is
  straightforward. Claim (3) is immediate, as is (4) because
  $G$ is compact and any finite-dimensional representation is
  continuous by definition.
\end{proof}

By \emph{countably-dimensional vector bundle} we shall always mean a
fiber bundle with structure group $O(\bR^\infty)$ (with the above
topology) and fiber $\bR^\infty$.

\begin{definition}
  Let $\fX$ be a topological stack. A \emph{universal vector
    bundle} on $\fX$ is a countably-dimensional vector bundle $\fH \to
  \fX$ such that any other (finite or countably dimensional) vector
  bundle $\fE \to \fX$ admits a complemented isometric embedding $\fE \hookrightarrow
  \fH$.
\end{definition}

A universal vector bundle $\fH$ has two ``absorption properties'': the
tensor product $\fH \otimes \bR^{\infty}$ is isometrically isomorphic
to $\fH$ and so is the sum $\fH \oplus \fE$ with any
countably-dimensional vector bundle $\fE$.

\begin{proposition}\label{univvectbund}
  If $\fX$ is a local quotient stack, then there exists a universal
  vector bundle $\fH \to \fX$. Furthermore, it is unique up to
  isometry. For any other vector bundle $\fE \to \fX$, the space of
  isometric embeddings $\Mon(\fE; \fH)$ is weakly contractible. The
  pullback along any representable morphism of local quotient stacks
  is also a universal vector bundle.
\end{proposition}

\begin{remark}
  The local quotient hypothesis is necessary. An example of a
  differentiable stack which is not a local quotient and does not
  admit a universal vector bundle is $\ast \hq \bR$ because the set
  of equivalence classes of orthogonal $\bR$-representations is
  uncountable.
\end{remark}

If countably-dimensional vector bundles are replaced by (separable)
Hilbert space bundles, then the analogue of Proposition
\ref{univvectbund} is proven in \cite{FHT}, p. 57 ff.  Due to Lemma
\ref{functionalanalysis}, (2), any countably-dimensional vector bundle
can be completed to a Hilbert space bundle. When one completes the
universal countably-dimensional bundle, one obtains the universal
Hilbert bundle. Hence Proposition \ref{univvectbund} can be
interpreted as the statement that the universal Hilbert-bundle on a
local quotient stack admits a ``countable substructure''.

To prove Proposition \ref{univvectbund}, the arguments of loc. cit. carry
over to our situation without essential change, with one exception:
the construction of a universal vector bundle on $X \hq G$ needs to be adjusted. 

Let $G$ be a compact Lie group, let $L^2 (G)$ be the regular
representation of $G$ and $L^2(G)_{fin} \subset L^2 (G)$ be the space
of $G$-finite vectors. It is an $L^2$-dense, countably-dimensional
subspace (see \cite{BroetDie}, Theorem 5.7). Moreover, it contains any
irreducible $G$-representation with finite multiplicity. Thus
$L^2(G)_{fin} \otimes \bR^{\infty}$ contains any countably-dimensional
$G$-representation.

\begin{lemma}\label{local}
  If $X$ is a paracompact space and $G$ is a compact Lie group acting
  continuously on $X$, then the vector bundle $X \times_{G}
  (L^2(G)_{fin} \otimes \bR^{\infty}) \to X \hq G$ is a universal
  vector bundle.
\end{lemma}

\begin{proof}
  This is only a slight modification of well-known results. First, let
  $G$ be the trivial group and let $\pi:V \to X$ be a vector bundle.
  Choose a countable cover $\{U_i\}_{i \in \bN}$ of $X$,
  trivializations $(\pi, \phi_i): V|_{U_i} \to U_i \times
  \bR^{\infty}$, and a locally finite partition of unity $(\lambda_i)$
  subordinate to the chosen cover.  Then the map $\phi: V \to X \times
  \bR^{\infty \times \infty}$ given by
\[
\phi(v) := \left( \pi(v), \sqrt{\lambda_1(\pi(v))} \phi_1 (v) ,
\sqrt{\lambda_2(\pi(v))} \phi_2 (v), \sqrt{\lambda_3(\pi(v))} \phi_3
(v), \ldots \right)
\]
is clearly an isometric complemented embedding.

For nontrivial $G$, we argue as follows.  Let $V \to X$ be a
$G$-equivariant vector bundle and let $j:V \hookrightarrow X \times
\bR^{\infty}$ be an embedding as constructed above.  By Frobenius
reciprocity \cite{BroetDie}, p. 144, there is a continuous isomorphism
$\Hom(W,U) \cong \Hom_G(W;U \otimes L^2 (G)_{fin})$ for all
countably-dimensional $G$-modules $W$ and all countably-dimensional
vector spaces $U$.  Use this to extend $j$ to a $G$-equivariant
embedding $j \dash : V \hookrightarrow X \times \bR^{\infty} \otimes
L^2(G)_{fin}$.
\end{proof}

Proposition \ref{univvectbund} now follows from Lemma \ref{local} by
the same arguments as used in \cite{FHT}.

\subsection{The Whitney embedding theorem for local quotient stacks}

\begin{proposition}\label{whitney}
  Let $f: \fX \to \fY$ be a proper representable morphism between
  differentiable local quotient stacks. Let $\pi:\fH \to \fY$ be a universal vector
  bundle. Then there exists a fat embedding $\fX \hookrightarrow
  \fH$ over $\fY$. More precisely, there exists a
  countably-dimensional vector bundle $q:\fE \to \fX$ with zero
  section $s$ and an open embedding $j:\fE \hookrightarrow
  \fH$ such that $\pi \circ j \circ s = f $. Moreover, the space
  of such embeddings is contractible.
\end{proposition}

\begin{proof}
  First assume that $\fY = Y \hq G$, where $G$ is a compact Lie group
  acting on $Y$ with finite orbit type, i.e. the number of conjugacy
  classes of isotropy subgroups is finite. Then $\fX \cong X \hq G$
  and $f$ is represented by a $G$-equivariant map $X \to Y$ (compare
  proof of Lemma \ref{fht}).  Because $f$ is proper, the action on $X$
  also has finite orbit type.  Mostow showed (\cite{most}, p. 444 f)
  that there exists a finite-dimensional $G$-representation $V$ and a
  $G$-equivariant embedding $X \hookrightarrow Y \times V$.  The vector bundle $Y
  \times_{G} V$ admits an isometric embedding into the universal
  bundle $Y \times_{G} L^2 (G)_{fin} \otimes \bR^{\infty}$.  There
  exist $G$-equivariant tubular neighborhoods; the choice of one gives
  a fat embedding. The space of all $G$-equivariant tubular
  neighborhoods is contractible.  A variant of the Eilenberg swindle
  (compare the proof Lemma A.35 in \cite{FHT}) shows that the space of
  all fat embeddings $\fX \to \fH$ over $\fY$ is contractible.

  If $\fY$ is a local quotient stack then we can glue these locally
  defined fat embeddings by the following procedure. Choose open
  substacks $W_i \hq G_i = \fW_i \subset \fU_i = U_i \hq G_i$, $i \in
  \bN$ of $\fY$ such that $W_i \subset U_i$ is relatively compact and
  such that the collection of all $\fW_i$ covers $\fY$.  Then the
  $G_i$-action on $W_i$ is of finite orbit type and therefore the
  space of fat embeddings $\fX \times_{\fY} \fW_i \to \fH|_{\fW_i}$ is
  contractible (in particular, nonempty).

  For any finite nonempty $S \subset \bN$ let $\fW_S$ be the
  intersection (alias fibered product over $\fX$) of all $\fW_i$, $i
  \in S$.  Being an open substack of some $\fW_i$, $\fW_S$ is also a
  global quotient stack.  We have seen that the space $\sF_S$ of all
  fat embeddings $\fX \times_{\fY} \fW_S \hookrightarrow \fH|_{\fW_S}$
  over $\fW_S$ is contractible for all finite nonempty subsets $S
  \subset \bN$.  Let $\Delta_S \subset \bR^S$ denote the $n$-simplex
  spanned by $S$.  Because the spaces $\sF_S$ are all contractible, by
  induction on $|S|$ it is possible to choose maps $h_S: \Delta_S \to
  \sF_S$ satisfying the compatibility conditions that whenever $T
  \subset S$ then $h_S|_{\Delta_T}=r_{T,S} \circ h_T$, where $r_{T,S}:
  \sF_T \to \sF_S$ is the obvious restriction map.

  Now let $\{ \lambda_i \}$ $(i \in \bN)$ be a locally finite
  partition of unity subordinate to the covering $\{ \fW_i \}$.  For
  each $x \in \fX$ let 
  \[
  S(x) := \{i \in \bN \:\: | \:\: f(x) \in \supp (\lambda_i)\}.
  \]
  For $x \in \fX$, we put
  \[
  j(x) := h_{S(x)}\left(\sum_{i\in S(x)} \lambda_{i}(f(x)) \{i\}
  \right)(f(x)) \in \fH.
  \]
  Using that $\{\lambda_i\}_{i \in \bN}$ is locally finite and the
  compatibility of the maps $\{h_S\}$, one observes that $j$ is continuous
  and a fat embedding. The contractibility follows from a similar
  argument.
\end{proof}

\subsection{Constructing the Pontrjagin-Thom map}

Let $f: \fX \to \fY$ be a proper representable morphism of local
quotient stacks, let $\fH \to \fY$ be a universal vector bundle, $\fE
\to \fX$ a vector bundle and $j: \fE \to \fH$ be a fat embedding over
$\fY$, as provided by Proposition \ref{whitney}.  Let $Y \to \fY$ be
an atlas and $X \to \fX$ the induced atlas for $\fX$.  Similarly, we
have atlases $E \to \fE$ and $H \to \fH$. When we use these atlases to
present the homotopy types, we get a diagram
\[
\xymatrix{
\Ho(\fE) \ar[d] \ar[r]^{\Ho(j)} & \Ho(\fH) \ar[d] \\
\Ho(\fX) \ar@/^/[u]^{s} \ar[r]^{\Ho(f)} & \Ho(\fY),
}
\]
where the vertical downward maps are countably-dimensional vector
bundles, the left vertical upward map is the zero section, $\Ho(j)$ is
an open embedding and $\Ho(f)$ is a proper map.  This square commutes
(for either choice of the left vertical arrow).

The bundle $\Ho(\fH)$ is the pullback of $\fH$ along the map
$\Ho(\fY)$, therefore it is a universal vector bundle on $\Ho(\fY)$;
hence $\Ho(\fH) \cong \Ho(\fY) \times \bR^{\infty}$ (this is of course
not true before taking homotopy types).  The proofs of Propositions
\ref{whitney} and \ref{univvectbund} show that any $y \in \Ho(\fY)$
has a neighborhood $U$ such that $\Ho(j)\circ s$ embeds
$\Ho(f)^{-1}(U)$ into a finite-dimensional subbundle of $\Ho(\fY)
\times \bR^{\infty}$. After passing to a smaller neighborhood and
moving by an isotopy of embeddings, we can assume that
\begin{equation}\label{u}
  \Ho(j) \circ s :\Ho(f)^{-1}(U) \hookrightarrow  U \times \bR^{n_0} \subset \Ho(\fY) \times \bR^{\infty}
\end{equation}
for some $n_0$ large enough.  If $n \geq n_0$, then each space
$V_{U,n}:=\Ho(j)^{-1}(U \times \bR^n)$ is the total space of a vector
bundle on $\Ho(f)^{-1}(U)$.  The rank of $V_{U,n}$ is equal to $d+n$,
where $d$ is the codimension of $f$.  Clearly $V_{U,n+1} \cong V_{U,n}
\oplus \bR$. Also $V_{U,n} = V_{U \dash ,n}$ on the intersection $U
\cap U \dash$. Therefore, the $V_{U,n}$ assemble to a stable vector
bundle $V$ on $\Ho(\fX)$.

\begin{lemma}
  The stable vector bundle $V$ is the stable normal bundle $\nu(f)$
  (pulled back to $\Ho(\fX)$).
\end{lemma}

\begin{proof}
  Let $N$ be a manifold, $N \to \fY$ a map transverse to $f$ and let
  $N \dash \subset N $ be a relatively compact open submanifold. Let
  $f_N: M = N \times_{\fY} \fX \to N$ be the induced map and similarly
  for $N \dash$. Choose a lift $N \to \Ho(\fY)$ of $N \to \fY$ by the
  universal property of the homotopy type. The fibered product
  $\Ho(\fY) \times_{\fY} \fX$ is a model for $\Ho(\fX)$ and it is easy
  to see that
\[
\xymatrix{
M \ar[r]^-{h} \ar[d]^{f_N} & \Ho(\fX) \ar[d]\\
N \ar[r] & \Ho(\fY)\\
}
\]
is a pullback diagram. Clearly, the relatively compact $N \dash$ maps
into some $U \subset \Ho(\fY)$ as in \ref{u}.

It is clear that the pullback $h^* V_{U,n}$ is (canonically)
isomorphic to $\nu(f_{N \dash}) \oplus \bR^n$. This is precisely the
same statement as saying that $V$ is isomorphic to the homotopical
realization of the stable normal bundle of $f$. 
\end{proof}

The collapse construction finally defines a map of spectra $\suspinf
U_+ \to \bTh (V)$ and by adjunction a map of spaces $U
\to \loopinf \bTh(V)$. By construction, these maps agree on intersections $U \cap U \dash$.
This finishes the
proof of Theorem \ref{pontthom}.

\section{The moduli stack of stable curves and graphs}\label{moduli-exposition}

The stack $\mnbar{g}{n}$ was first constructed in the algebraic
category by Deligne, Mumford and Knudsen (in \cite{DM} when $n=0$ and
\cite{Knud} for general $n$).  We will need only the associated
orbifold in the category of differentiable stacks. For more
information about $\mnbar{g}{n}$, we refer to the textbook
\cite{HarMor} or the article \cite{Edi}.

\subsection{The moduli stack of stable curves} A nodal curve is a
complete complex algebraic curve $C$ all of whose singularities are
nodal, i.e. ordinary double points. The open subset of smooth points
of $C$ will be denoted by $C^{sm}$. The \emph{arithmetic genus} of a
connected nodal curve is the dimension of the vector space $H^1 (C,
\cO_C)$.  Let $C_1, \ldots , C_k$ be the components of $C$, let $g_i$
be the genus of $C_i$ and let $r$ be the number of nodes of $C$. Then
the arithmetic genus is given by
\begin{equation}\label{genusformula}
g = \sum_{i=1}^{k} (g_i-1) + r+1.
\end{equation}
All nodal curves in this paper are understood to be connected.  Given
a finite set $P$, a $P$-pointed nodal curve is a nodal curve $C$ with
an embedding of $P$ into $C^{sm}$. Such a curve is
\emph{stable} if its automorphism group is finite.  This means that
the Euler characteristic of each component of $C^{sm} \smallsetminus
P$ is negative, or equivalently, $C$ does
not contain an irreducible component which is a projective line with
fewer than $3$ marked points and nodes or an elliptic curve with no
marked points or nodes.

The stack $\mnbar{g}{P}$ is the lax sheaf of groupoids on the site
of schemes over $\bC$ in the \'etale topology which is given by:
\begin{enumerate}
\item The objects of $\mnbar{g}{n}(X)$ are pairs $(E
  \stackrel{\pi}{\to} X,j:X\times P \hookrightarrow E)$, where $\pi$ a
  proper morphism all of whose geometric fibers are reduced connected
  nodal curves of genus $g$, and $j$ is an embedding over $X$, and
  each fiber is a $P$-pointed stable nodal curve.  Such a triple a
  \emph{family of pointed curves over $X$}.
\item An isomorphism of families of pointed stable curves is an
isomorphism of schemes over $X$ which respects the embedding $j$.
\end{enumerate}
Deligne-Mumford-Knudsen \cite{DM}, \cite{Knud} constructed a smooth \'etale
atlas for $\mnbar{g}{P}$ in the category of schemes over
$\operatorname{spec}\bC$.  In the complex analytic category an
orbifold atlas is given by the degeneration spaces of Bers
\cite{Bers}, and another was constructed in \cite{SalRob}.  The
complex dimension of $\mnbar{g}{P}$ is $3g - 3 + |P|$.  An important
property of this stack is that its coarse moduli space is compact.

The symmetric group $\Sigma_P$ acts on $\mnbar{g}{P}$ by permuting the
marked points; thus $\Sigma_P$ acts on $\Ho(\M_{g,P})$.

\subsection{Stable graphs}

Following \cite{GK}, we introduce \emph{stable graphs} as a
combinatorial tool for working with the stratification of $\M_{g,n}$
and keeping track of iterations of gluing morphisms.

A \emph{graph} $\Gamma$ consists of a finite set $\ver (\Gamma)$ of
vertices, a finite set $H(\Gamma)$ of half-edges, together with an
involution $\sigma$ on $H(\Gamma)$ and a map $\tau: H(\Gamma) \to \ver
(\Gamma)$.  The set of half-edges incident at a vertex $v$ is
$\tau^{-1}(v)$.  An \emph{edge} is a free orbit of $\sigma$ and the
endpoints of an edge are the vertices that its half-edges are incident
at. We write $E(\Gamma)$ for the set of edges.  The fixed points of
$\sigma$ are the \emph{legs} of $\Gamma$.

A \emph{stable graph} is a graph $\Gamma$, together with a function
$g: \ver (\Gamma) \to \bN_{\geq 0}$ satisfying $g(v) > 0$ if the
valence $v$ is less than 3 and $g(v) > 1$ if the valence is 0.  The
\emph{genus} of a connected stable graph is defined to be
\[
g(\Gamma) = \sum_{v \in \ver (\Gamma)} (g(v)-1) + |E(\Gamma)|+1.
\]
We will need stable graphs equipped with an additional piece of data:
a subset $U$ of the univalent vertices of $\Gamma$ and for each $u\in
U$ a point $[F_u]: \ast \to \mnbar{g(u)}{1}$ corresponding to a stable
curve $F_u$. The vertices of $U$ are called \emph{pointed vertices}
and the stable curves $F_u$ are \emph{decorations}.  An
\emph{automorphism} of a stable graph consists of two bijections of
the sets $\ver (\Gamma)$ and $H(\Gamma)$ compatible with $\sigma$,
$\tau$ and the function $g$, fixing the legs pointwise, and sending
pointed vertices to pointed vertices with identical decorations.

Given a stable graph $\Gamma$ and an edge $e$ we produce three new
stable graphs.  The graph $\Gamma \smallsetminus e$ is obtained by
deleting the edge $e$ (if the resulting graph contains a bivalent
vertex of genus $0$, then replace it by a single edge or half-edge, as
appropriate). The graph $\Gamma | e$ is obtained by cutting $e$ into
two legs.  The graph $\Gamma / e$ is obtained by contracting the edge
$e$; if the endpoints of $e$ are two distinct vertices then one
identifies them and adds their genera, and if the endpoints are the
same then one increases the genus of that vertex by one.  More
generally, if $K$ is a set of edges then we construct
$\Gamma\smallsetminus K$, $\Gamma | K$, and $\Gamma/K$ by iterating
the above constructions.

Given a stable graph $\Gamma$, we define stacks
\begin{align*}
  & \M(\Gamma)  := \left( \prod_{v \in \ver (\Gamma) \smallsetminus U}
    \mnbar{g(v)}{\leg(v)} \times \prod_{v\in U}  \ast \right), \\
  & \M((\Gamma)) := \M(\Gamma)\hq \Aut(\Gamma).
\end{align*}
The substacks $\fM(\Gamma) \subset \M(\Gamma)$ and
$\fM((\Gamma))\subset \M((\Gamma))$ are defined analogously.  Observe
that as $\Gamma$ runs over the isomorphism classes of graphs (of type
$(g,n)$) with no pointed vertices, $\fM(\Gamma)$ runs over the open
strata of $\M_{g,n}$.  The decorations $[F_u]$ on the pointed univalent
vertices define an immersion $\M(\Gamma) \to
\M(\Gamma')$, where $\Gamma'$ is obtained from $\Gamma$ by replacing
all pointed vertices by ordinary vertices.  Note that there is a
canonical isomorphism $\M((\Gamma | K)) \cong
\M(\Gamma) \hq \operatorname{Stab}(K)$.

An edge $e$ determines a gluing morphism $\M(\Gamma) \to \M(\Gamma /
e)$ as follows.  Let $v$, $v'$ be the endpoints of $e$.  If $v$ and
$v'$ are distinct then this morphism is induced by the gluing morphism
$\mnbar{g(v)}{\tau^{-1}(v)} \times \mnbar{g(v')}{\tau^{-1}(v')} \to
\mnbar{g(v)+g(v')}{\tau^{-1}(v \sqcup v') \smallsetminus e}$ defined
by gluing curves together at the marked points corresponding to the
half-edges of $e$.  If $v=v'$ then it is induced by the gluing
morphism $\mnbar{g(v)}{\tau^{-1}(v)} \to
\mnbar{g(v)+1}{\tau^{-1}(v)\smallsetminus e}$ again defined by gluing
marked points together to form a node.  More generally, a set $K$ of
edges determines a gluing morphism
\[
\widetilde{\xi}_K: \M(\Gamma) \to \M(\Gamma/K),
\]
and if $K$ is $\Aut(\Gamma)$-invariant then this morphism is
equivariant and hence descends to a morphism
\[
\xi_K: \M((\Gamma)) \to \M((\Gamma/K)).
\]
Let $\ast_{g,n}$ denote the stable graph consisting of a single genus $g$
vertex and $n$ legs; its automorphism group is trivial.  Contracting
all edges of $\Gamma$ gives a map $\xi_{E(\Gamma)} : \M ((\Gamma)) \to
\fM((\ast_{g,n})) = \mnbar{g}{n}$.  The restriction of
$\xi_{E(\Gamma)}$ to $\fM(\Gamma)$ is the inclusion of the open
stratum labeled by $\Gamma$.

Given a leg $h$, there is a forgetful morphism $\M(\Gamma) \to
\M(\Gamma \smallsetminus h)$ given by forgetting the marked point on a
stable curve corresponding to the leg $h$.  (As usual, if we forget a
marked point on a genus zero component with only two additional marked
points or nodes then we collapse that component to a node). More
generally, given a set of edges $K$, there
is a forgetful morphism
\[
\widetilde{\pi}_K: \M(\Gamma) \to \M(\Gamma \smallsetminus K);
\]
if $K$ is $\Aut(\Gamma)$-invariant, then it descends to
\[
\pi_K: \M((\Gamma)) \to \M((\Gamma \smallsetminus K)).
\]

\subsection{Vector bundles on $\M(\Gamma)$ and stripping and splitting }

On $\mnbar{g}{n}$ there are complex line bundles $\widetilde{L}_1,
\ldots, \widetilde{L}_n$ whose fibres at a given curve are the tangent
spaces at each of the marked points.  Hence a half-edge $h$ of
$\Gamma$ determines a complex line bundle $\widetilde{L}_h$ on
$\M(\Gamma)$.  More generally, a set $H$ of half-edges determines a
vector bundle $\widetilde{L}_H = \oplus_{h\in H} \widetilde{L}_h$, and
this bundle is $\Aut(\Gamma)$-equivariant whenever $H$ is
$\Aut(\Gamma)$-invariant.  Assuming $H$ is invariant,
$\widetilde{L}_H$ is classified by a map $\widetilde{L}_H: \M(\Gamma)
\to BU(1)^H$ which descends to homotopy orbits to give a map
\begin{align*}
L_H: \M((\Gamma)) \to B(U(1)^H \rtimes \Aut(H,\Gamma)) \\
= E\Aut(H,\Gamma)\times_{\Aut(H,\Gamma)} BU(1)^H,
\end{align*}
where $\Aut(H,\Gamma)$ is the group of permutations of $H$ which are
induced by automorphisms of $\Gamma$.

Given a set $H$ of half-edges, let $E_H$ denote the set of edges which
contain elements of $H$.  The following proposition will be used in
the proof of Theorem \ref{mainthm}.
\begin{proposition}\label{strip-and-split-prop}
  (i) If $H$ a set of half-edges consisting of an $\Aut(\Gamma)$-orbit
  of edges between ordinary vertices (not necessarily distinct, having
  genera $g_1$ and $g_2$) then the $\Aut(\Gamma)$-equivariant map
  \begin{equation}\label{stripsplit1}
    \fM(\Gamma) \stackrel{\widetilde{\pi}_{E_H} \times \widetilde{L}_H}{\longrightarrow} \fM(\Gamma \smallsetminus E_H) \times BU(1)^H
  \end{equation}
  is a homology isomorphism in degrees $*\leq
  \operatorname{min}\{g_1/2-1, g_2/2-1\}$, and hence it induces an
  isomorphism in this range on the homotopy orbits.

  (ii) If $H$ is an $\Aut(\Gamma)$-orbit of half-edges $h$ (incident
  at an ordinary vertex of genus $g_1$) such that $\sigma(h)$ is
  incident at a univalent pointed vertex then the
  $\Aut(\Gamma)$-equivariant map
  \begin{equation}\label{stripsplit2}
    \fM(\Gamma) \stackrel{\widetilde{\pi}_{E_H} \times \widetilde{L}_H}{\longrightarrow}
    \fM(\Gamma \smallsetminus E_H) \times BU(1)^H,
  \end{equation}
  is a homology isomorphism in degrees $* \leq g_1/2-1$, and hence
  it induces a homology isomorphism in this range on homotopy
  orbits.
\end{proposition}
\begin{proof}
  B\"odigheimer and Tillmann proved in \cite{BoeTill} that
  Harer-Ivanov stability \cite{Har}, \cite{Iv} implies that the
  $\Sigma_P\times \Sigma_Q$-equivariant ``stripping-and-splitting'' map
  \[
  \fM_{g,P\sqcup Q} \stackrel{\pi_Q\times L_Q}{\longrightarrow}
  \fM_{g,P} \times BU(1)^Q
  \]
  is a homology isomorphism in degrees $*\leq g/2-1$.  The proposition
  is a straighforward application of this. The statements about
  homotopy quotients follow from a standard argument with the
  Leray-Serre spectral sequence.
\end{proof}

\begin{remark}
  Here is a proof of the theorem in \cite{BoeTill}, easier than the
  original one. The stripping and splitting map is the middle vertical
  arrow in the following diagram (whose rows are homotopy-fibrations)
  \[
  \xymatrix{
    \fM_{g,P,\vec{Q}} \ar[r] \ar[d] &  \fM_{g,P\sqcup Q} \ar[d] \ar[r] & BU(1)^Q \ar@{=}[d]\\
    \fM_{g,P} \ar[r] & \fM_{g,P} \times BU(1)^Q \ar[r] &
    BU(1)^Q,}
  \]
  where $\fM_{g,P,\vec{Q}}$ is the moduli stack of smooth curves of
  genus $g$ with $|P|$ marked points and $|Q|$ additional marked
  points equipped with a nonzero tangent vector.  The left vertical arrow
  is a homology equivalence in the stable range by Harer-Ivanov
  stability.  The base space is simply-connected, so both fibrations
  are simple.  Thus a straightforward application of the Leray-Serre
  spectral sequence finishes the proof.
\end{remark}

\subsection{The irreducible components of the boundary}\label{boundary-section}

Let $\sD$ denote the set of irreducible components of the boundary of
$\M_{g,n}$.  The elements of $\sD$ are indexed by the (isomorphism
classes of) stable graphs of genus $g$ with $n$ legs, a single edge
$e=\{h_1, h_2\}$, and no pointed vertices.  We call such a stable
graph \emph{elementary}.  The elementary graph consisting of a single
vertex with a loop and $n$ legs is denoted $\Gamma_{irr}$, and it
corresponds to the locus of curves with a non-separating node.  The
other boundary components correspond to elementary graphs with two
vertices; they are indexed by the partition of $g$ between the two
vertices and the subset $P\subset \{1, \ldots, n\}$ of legs incident
at the vertex of lesser genus.

Given $\alpha \in \sD$, let $\Gamma_\alpha$ denote the corresponding
elementary graph.  The sole edge of $\Gamma_\alpha$ determines a
gluing morphism
\[
\xi_\alpha: \M((\Gamma_\alpha)) \to \M_{g,n}
\]
which is a complex codimension 1 immersion whose image is precisely the
boundary component $\alpha$.  The elementary graphs with two vertices
and a nonzero number of legs incident
at the vertex of smaller genus correspond to the boundary components
which have no self-intersections, so the gluing morphisms for these
are embeddings.  We will only be interested in the
self-intersecting boundary components.  Let $\sD^+ \subset \sD$
denote the set of boundary components which have nontrivial
self-intersections.  See Figure \ref{simplegraphs}.
\begin{figure}
\begin{center}
\input{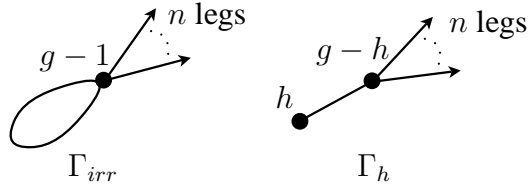}
\end{center}
\caption{\label{simplegraphs}Elementary graphs indexing the boundary
  components of $\sD^+$; i.e. those which have nontrivial
  self-intersections.}
\end{figure}

Let $\Gamma_\alpha$ be an elementary graph with edge $e=\{h_1,h_2\}$.
The gluing morphism $\xi_\alpha = \xi_e : \M((\Gamma_\alpha)) \to
\M_{g,n}$ has normal bundle is $L_{h_1} \otimes L_{h_2}$ (see
\cite{HarMor} p.101).  Note that if $\alpha=(irr)$ then there is an
automorphism swapping the two half-edges, so one only has this tensor
decomposition after pulling back to $\M(\Gamma_{irr})$.  Thus the
structure group of the normal bundle of $\xi_\alpha$ can be uniformly
written as $T(2)\rtimes \Aut(\Gamma_\alpha)$.

\section{The Pontrjagin-Thom maps for $\mnbar{g}{n}$ in homology}\label{maintheorem-section}

\subsection{Statement of results and outline of proof}

In this section we will state and prove our main result in full
generality.  But first we need to set up some terminology and
notation.  Throughout this section we will drop the notational
distinction between stacks and their homotopy types and between stack
quotients and the corresponding homotopy quotients.  We fix a genus
$g$ and number $n$ of marked points throughout.

For any boundary component $\alpha\in \sD^+$, the complex codimension
1 immersion $\xi_\alpha : \M((\Gamma_\alpha)) \to \mnbar{g}{n}$ has a
normal bundle $\nu(\alpha)$.  By the discussion of section
\ref{boundary-section}, the structure group of $\nu(\alpha)$ is
equipped with a distinguished lift to $T(2)\rtimes
\Aut(\Gamma_\alpha)$. Hence the Pontrjagin-Thom construction from Theorem \ref{pontthom} provides maps
\begin{align*}
\Phi^1_\alpha:\: & \mnbar{g}{n} \to \Q\M(\Gamma_\alpha)^{\nu(\alpha)} \to \Q B(T(2)\rtimes \Aut(\Gamma_\alpha))^V,\\
\Phi^0_\alpha:\: & \mnbar{g}{n} \to \Q\M(\Gamma_\alpha)^{\nu(\alpha)} \to \Q BU(1)^V,
\end{align*}
(the difference is whether or not we use the lifted structure group).
We will prove that a product of several maps of the above type is
surjective on homology (with field coefficients) in a range of
degrees.  We consider both, $\Phi^1_\alpha$ and $\Phi^0_\alpha$,
because the target of the former has larger homology, while the range
of surjectivity is often larger for the latter.

Fix a subset $A \subset \sD^+$ and a function $\ell: A \to \{0,1\}$,
such that $\ell(irr)=1$ if $(irr) \in A$.  For each $\alpha \in A$,
define
\[
G_\alpha = \left\{ \begin{array}{ll}
    N(2) & \alpha = (irr)  \\
    T(2) & \alpha \neq (irr), \ell(\alpha)=1 \\
    U(1) & \ell(\alpha) = 0
\end{array} \right.
\]
Thus the target of $\Phi_\alpha^{\ell(\alpha)}$ is $\Q BG_\alpha^V$.
Set $g_{irr} = 1$, and for any other elementary graph $\Gamma_\alpha$
set $g_\alpha$ to be the lesser of the genera of the two vertices.  We
define an \emph{$A$-partition} of $g$ to be a set
$\vec{m}:=(m_{\alpha})_{\alpha \in A}$ of nonnegative integers such
that $r := g - \sum_{\alpha \in A} m_\alpha g_\alpha$ is nonnegative.
Given an $A$-partition $\vec{m}$, we set
\begin{align*}
c(A,\ell, \vec{m})  := & \min \left\{r/2-1, \quad m_\alpha/2
  \quad (\alpha \in A),\right. \\ 
   &  \left. \quad g_\alpha/2 -1 \quad (\alpha \in A
  \mbox{ with } \ell(\alpha)=1 \mbox{ and } \alpha\neq irr) \right\}, \\
c(A,\ell)  := & \max \{c(A, \ell , \vec{m}) \quad | \quad \mbox{$\vec{m}$
  an $A$-partition of $g$}\}.
\end{align*}

\begin{theorem}\label{mainthm}
Given a pair $(A\subset \sD^+, \ell: A \to \{0,1\})$, the map
\[
\prod_{\alpha \in A} \Phi_{\alpha}^{\ell(\alpha)}: \Ho(\mnbar{g}{n})
\to \prod_{\alpha \in A} \Q BG_\alpha^V
\]
is surjective on ordinary homology with field coefficients in degrees
$*\leq c(A,\ell)$.
\end{theorem}

\begin{remark}
\begin{enumerate}
\item The special case $A =\{(irr)\}$ gives case (1) of Theorem
  \ref{mainthmspecial}.  Taking $A=\{h\}$ and $\ell(h)=1$ or $0$ give
  the other two cases.
\item As we will see in the proof of this theorem, the homology
  surjectivity comes from boundary components that have high numbers
  of self-intersections.  Thus Pontrjagin-Thom maps for the boundary
  components which are embedded (rather than immersed) factor as
\[
\mnbar{g}{n} \to BT(2)^V \to \Q(BT(2)).
\]
Such maps cannot be surjective in homology in a range because the
second map is not.  This is why the theorem refers only to
self-intersecting boundary components.
\end{enumerate}
\end{remark}

Here is an outline of the proof of Theorem \ref{mainthm}.  We consider
the following diagram:
\begin{equation}\label{main-diagram}
\xymatrix{\fM((\Gamma)) \ar[r]^{\xi_{E(\Gamma)}} \ar[d]^-{L_{H_0}} & \mnbar{g}{n}
  \ar[r]^-{\prod \Phi_\alpha^{\ell(\alpha)}} & \prod_{\alpha \in A} \Q BG_\alpha^V \\
  \prod_{\alpha \in A} B(G_\alpha \wr \Sigma_{m_\alpha})
  \ar[rr]^-{\prod \groupcompletion_\alpha} & & \prod_{\alpha \in A} \Q_{(m_\alpha)}
  (BG_\alpha)_+ \ar[u]^-{\prod \Q \inc}.  }
\end{equation}
Here $\fM((\Gamma))$ is the open stratum in $\mnbar{g}{n}$ determined
by a certain stable graph $\Gamma$.  Going around counter-clockwise:
the map $L_{H_0}$ is the classifying map for a vector bundle
determined by a certain set $H_0$ of half-edges of $\Gamma$; the maps
$\groupcompletion_\alpha$ are components of the group completion map
appearing in the Barratt-Priddy-Quillen-Segal Theorem; the last map,
$\Q \inc$, is induced by the inclusion of the zero section into the Thom
space.  We will define each of these maps in detail and show that they
are each surjective on homology in certain ranges of degrees, and that
the above diagram commutes up to homotopy.

\subsection{Definition of the stable graph $\Gamma$}
Fix an $A$-partition $\vec{m}$ such that $c(A,\ell, \vec{m})$ is
maximal and construct $\Gamma$ as follows.  There is a vertex $v$ of
genus $r$, $n$ legs incident at $v$, and $m_{irr}$ loops at $v$ (if
$(irr) \in A$).  For each $\alpha \in A \smallsetminus (irr)$ there
are $m_{\alpha}$ additional vertices of genus $g_{\alpha}$, each of
which is connected to $v$ by a single edge.  For each $\alpha$ with
$\ell(\alpha)=0$ the univalent vertices of genus $g_\alpha$ are
pointed (all with the same decoration).  See figure \ref{testgraph}.

The automorphism group of $\Gamma$ is
\[
\Aut(\Gamma) \cong \prod_{\alpha \in A} \Aut(\Gamma_\alpha) \wr
\Sigma_{m_\alpha}.
\]
Let $H_0$ denote the set of all half-edges incident at ordinary
vertices.  One then has
\[
U(1)^{H_0} \rtimes \Aut(\Gamma) \cong
\prod_{\alpha \in A} G_{\alpha}\wr \Sigma_{m_\alpha}.
\]
\begin{figure}
\input{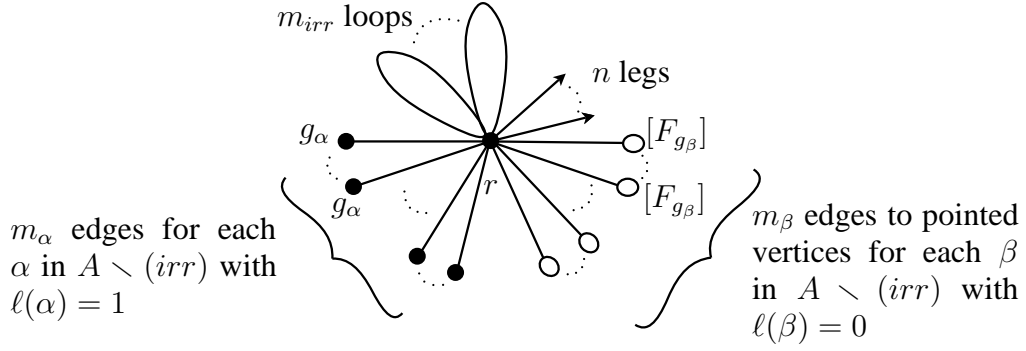}
\caption{\label{testgraph} The test graph $\Gamma$ for the proof of Theorem
  \ref{mainthm}.}
\end{figure}

\subsection{The classifying map of $L_{H_0}$}\label{harerivan}

Recall that $H_0$ is the set of half-edges of $\Gamma$ which are
incident at ordinary vertices; this determines a vector bundle
$L_{H_0}$ on $\M((\Gamma))$ with structure group $U(1)^{H_0}\rtimes
\Aut(\Gamma) = \prod_{\alpha \in A} G_\alpha \wr \Sigma_{m_\alpha}$.

\begin{lemma}
  The classifying map $L_{H_0}: \M((\Gamma)) \to \prod_{\alpha \in A}
  B(G_\alpha \wr \Sigma_{m_\alpha})$ is surjective on homology in
  degrees
  \[
  * \leq \operatorname{min} \left\{ r/2-1 ,\:\: g_{\alpha}/2-1 \:\:
    (\mbox{for $\alpha$ such that $\ell(\alpha)=1$}) \right\}.
  \]
\end{lemma}
\begin{proof}
  Let $E_{irr}$ denote the set of loops at the central vertex, and for
  $\alpha\neq (irr)$ let $E_\alpha$ denote the set of edges of
  $\Gamma$ which meet an outer vertex of genus $g_\alpha$.  Let
  $H_\alpha$ denote the set of half-edges lying in $E_\alpha$ which
  are incident at an ordinary vertex.  There is an
  $\Aut(\Gamma)$-invariant decomposition $H_0= \coprod_{\alpha \in A}
  H_\alpha$, The classifying map for the bundle $\widetilde{L}_{H_0}$
  on $\M(\Gamma)$ factors as
\begin{align*}
  \M(\Gamma) & \to
     \M(\Gamma \smallsetminus E_{\alpha_1}) \times
     BU(1)^{H_{\alpha_1}} \\
    & \longrightarrow \M(\Gamma \smallsetminus E_{\alpha_1} \cup
                 E_{\alpha_2}) \times BU(1)^{H_{\alpha_1}\cup H_{\alpha_2}} \\
    & \to \cdots \to \M(\Gamma \smallsetminus
                  \cup_{\alpha \in A} E_\alpha) \times BU(1)^{H_0} \\
    & \stackrel{\operatorname{proj}}{\longrightarrow} BU(1)^{H_0},
\end{align*}
where the first sequence of arrows are the stripping-and-splitting
maps of Proposition \ref{strip-and-split-prop}, and the final arrow is
simply projection.  All of these maps are $\Aut(\Gamma)$-equivariant,
and the final map admits an equivariant section by choosing a fixed
point in $\M(\Gamma \smallsetminus \cup_\alpha E_\alpha)$.  One
obtains the classifying map of the bundle $L_{H_0}$ on $\M((\Gamma))$
by passing to homotopy orbits.  The sequence of
stripping-and-splitting maps induce homology isomorphisms in the
stated range of degrees on homotopy orbits and the projection induces
a homology epimorphism on the homotopy orbits.
\end{proof}

\subsection{Symmetric groups and group completion}\label{groupcomthm}

We now discuss the map $\groupcompletion$ occurring in diagram
\eqref{main-diagram}. It is a special case of a general construction,
called ``group completion''.  Let $X$ be a connected space. There is
an $m$-fold covering
\[
E(\Sigma_{m-1}\times 1) \times_{\Sigma_{m-1}\times1} X^m \to E\Sigma_m
\times_{\Sigma_m} X^m
\]
induced by the index $m$ inclusion $\Sigma_{m-1} \times 1
\hookrightarrow \Sigma_m$.  The Becker-Gottlieb transfer is a map of
spectra
\[
\trf: \suspinf (E\Sigma_m \times_{\Sigma_m} X^m)_+ \to
\suspinf(E(\Sigma_{m-1}\times 1) \times_{\Sigma_{m-1} \times 1} X^m)_+
\]
which can be described (when $X$ is a manifold or local quotient
stack) as the Pontrjagin-Thom construction for the covering
projection.  The adjoint of the transfer is $E\Sigma_m
\times_{\Sigma_m} X^m \to \Q_{(m)}(E(\Sigma_{m-1}\times 1)
\times_{\Sigma_{m-1}\times 1} X^m)_+$, where $\Q_{(m)}$ denotes the
$m^{th}$ component.  The group completion map is then the composition
\[
\groupcompletion_m: E\Sigma_m \times_{\Sigma_m} X^m \stackrel{\trf}{\to}
\Q (E(\Sigma_{m-1}\times 1) \times_{\Sigma_{m-1}\times 1} X^m)_+ \to
\Q_{(m)}X_+,
\]
where the second map is induced by projecting onto the $m^{th}$
component of $X^m$.

The name stems from the following.  One can put a monoid structure on
the union
\[
\coprod_m E\Sigma_m \times_{\Sigma_m} X^m,
\]
and the maps $\{\groupcompletion_m\}$ assemble to a monoid map
$\groupcompletion: \coprod_m E\Sigma_m \times_{\Sigma_m} X^m \to \Q
X_+$.  The Barratt-Priddy-Quillen-Segal Theorem (see e.g.
\cite{Segcat}) asserts that this map is the homotopy-theoretic group
completion of the above monoid.

\begin{lemma}\label{homologicalstability}
  For any connected space $X$, the map $\groupcompletion_m: E \Sigma_m
  \times_{\Sigma_m} X^m \to \Q_{(m)} X_+$ induces a isomorphisms in
  homology with field coefficients\footnote{If the homology of $X$ is
    of finite type, then the coefficients can be arbitrary.} in
  degrees $* \leq (m-1)/2$.
\end{lemma}
\begin{proof}
  This is a well-known consequence of homology stability for symmetric
  groups (with twisted coefficients).  After choosing a basepoint in
  $X$ one has stabilization maps
  \[
  j_m: E \Sigma_{m} \times_{\Sigma_{m}} X^{m} \to E \Sigma_{m+1}
  \times_{\Sigma_{m+1}} X^{m+1}
  \]
  whose colimit is denoted by $E \Sigma_{\infty}
  \times_{\Sigma_{\infty}} X^{\infty}$. The induced map
  \[
 \groupcompletion_\infty: E \Sigma_{\infty} \times_{\Sigma_{\infty}}
X^{\infty} \to  \Q X_+
\]
is a homology isomorphism onto the basepoint component by the group
completion theorem (see e.g. \cite{McDuff-Segal}, \cite{Segconf}). Up
to a shift of component the stabilization map from $E\Sigma_m
\times_{\Sigma_m} X^m$ to the colimit followed by
$\groupcompletion_\infty$ agrees with $\groupcompletion_m$. The
stabilization map $j_m$ induces isomorphism in homology in degrees $*
\leq m/2 -1$, and epimorphism when $*\leq m/2$.\footnote{This has
probably been known for a long time. One proof can be found in
\cite{Hanb}, based on a result of \cite{Bet}.  Another proof is a
combination of Proposition 1.6 in \cite{HatWah} with the main result
of \cite{KanThu}.  The authors are not aware of a previously published
proof.}
\end{proof}
When $X=BG_\alpha$ and $m=m_\alpha$, the map $\groupcompletion_m$ of
this lemma is precisely the map $\groupcompletion_\alpha$ occuring in
Lemma \ref{main-diagram}, so the map
\[
\prod \groupcompletion_\alpha: \prod_{\alpha\in A} B(G_\alpha \wr
\Sigma_{m_\alpha}) \to \Q_{(m_\alpha)}(BG_\alpha)_+
\]
is a homology epimorphism in degrees $*\leq \operatorname{min}
\{m_\alpha/2 \:\:|\:\: \alpha \in A\}$.

\subsection{Homology of the Thom spectra and their infinite loop
spaces}\label{homthomspace}

The third arrow, $\prod \Q \inc: \prod_{\alpha \in A}\Q_{(m_\alpha)}
(BG_{\alpha})_+ \to \prod_{\alpha \in A}\Q (BG_{\alpha}^V)$, in the
counterclockwise composition in the diagram \eqref{main-diagram} is induced by
the inclusion of the zero section of the vector bundle $V$ over
$BG_\alpha$.  We show here that it is surjective in homology with
field coefficients.  This is the only part of the proof of Theorem
\ref{mainthm} where field coefficients are used seriously.
The homology surjectivity is immediate from the following two lemmata.

\begin{lemma}\label{thomiso}
The inclusions of the zero sections
\begin{align*}
BT(2) \hookrightarrow BT(2)^V, \\
BN(2) \hookrightarrow BN(2)^V, \\
BU(1) \hookrightarrow BU(1)^V
\end{align*}
induce surjections in homology with field coefficients.
\end{lemma}

\begin{lemma}\label{dyerlashof}
  If $f: X \to Y$ is a pointed map between pointed spaces which is
  surjective in homology with coefficients in a field $\bF$, then the
  induced map $\Q f: \Q X \to \Q Y$ is surjective on homology with
  coefficients in $\bF$.
\end{lemma}
Lemma \ref{dyerlashof} is a well-known fact which we discuss in
section \ref{infinite-loop-space-homology}.

\begin{proof}[Proof of Lemma \ref{thomiso}]
Let $\iota$ denote one of the three above inclusions.  We shall
prove the equivalent statement that $\iota^*$ is injective on
cohomology.  The composition of the Thom isomorphism followed by
$\iota^*$ is equal to multiplication by the Euler class $e(V)$ of
$V$, so it suffices to show that $e(V)$ is not a zero-divisor in
each of the three cases. For $BT(2)$, $BU(1)$ and $BN(2); \car(\bF)
\neq 2$, the computation is easy and well-known.

For $BN(2)$ and $\car(\bF)=2$, we argue as follows. The homogeneous space
$U(2) / N(2)$ is diffeomorphic to $\bR \bP^2$, so there is a fibration
\[
\bR \bP^2 \to BN(2) \stackrel{p}{\to} BU(2),
\]
which is simple because $BU(2)$ is simply-connected. Put $y_i := p^{*}
c_i \in H^{*}(BN(2);\bF)$.  Consider the Leray-Serre spectral sequence
for the fibration.  The cohomology of the real projective plane is
\[
H^*(\bR \bP^2; \bF) \cong \bF [w] / (w^3),
\]
where $w \in H^1 (\bR \bP^2; \bF)$.  One has $H^1(BN(2); \bF_2) \cong
\bF_2$, since $\pi_1 BN(2) \cong \bZ/2$, and hence the spectral
sequence collapses. Thus for $\car \bF =2$,
\[
H^*(BN(2); \bF) \cong \bF [w,y_1, y_2] / (w^3).
\]
The line bundle $V \to BN(2)$ is the tensor product of the determinant
line bundle $BN(2) \to BU(2) \stackrel{B\det}{\to} BS^1$ with the
signum line bundle $BN(2) \to B \bZ/2 \stackrel{B\inc}{\to} BS^1$ and
thus $c_1 (V) = y_1 + w^2$, which is not a zero
divisor in $\bF[w,y_1,y_2]/(w^3)$.
\end{proof}

\subsection{Homotopy commutativity of diagram
  \eqref{main-diagram}}\label{maindiagsect}

We first establish a general fact about Pontrjagin-Thom maps. We say
that two maps are \emph{weakly homotopic} if their restrictions to any
compact subset of the domain are homotopic.  Weakly homotopic maps
induce identical homomorphisms on homotopy groups and in any
generalized homology theory.

\begin{lemma}\label{transfer-lemma}
  Let $D \overset{f}{\looparrowright} M \overset{g}{\leftarrow} Z$ be
  a diagram of smooth manifolds such that
\begin{enumerate}
\item $f$ is a proper immersion,
\item $\mathrm{Im}(g) \subset \mathrm{Im}(f)$,
\item the projection $q: D\times_M Z \to Z$ is a finite sheeted covering.
\end{enumerate}
Let $h$ be the projection $D \times_M Z \to M$ and $\inc:D \to
D^{\nu(f)}$ the zero section. Then the diagram of spaces
\begin{equation}\label{transfer-pt-spectra-diagram}
\xymatrix{ Z \ar[r]^-{ g} \ar[d]^-{\trf_{q}}&
 M \ar[r]^-{\PT_f} & \Q D^{\nu(f)} \\
\Q (D\times_M Z)_+ \ar[rr]^-{\Q h}  & & \Q D_+
\ar[u]^-{\Q \inc}\\}
\end{equation}
commutes up to homotopy.  If $D$, $M$ and $Z$ are differentiable
stacks then the same statement is true except that the diagram is only
weakly homotopy commutative.
\end{lemma}
\begin{proof} Choose a map $j: D \to \bR^n$ such that $(f,j): D \to M\times \bR^n$
  is an embedding; it is proper because $f$ is proper.  The map
\begin{align*}
k: & D\times_M Z \hookrightarrow Z\times \bR^n \\
   & x \mapsto (q(x), j\circ  h(x))
\end{align*}
is therefore also a proper embedding.  Identifying $\nu((f,j))$ with
a tubular neighborhood of $D$ in $M\times \bR^n$, the inverse image
under $g\times \mathrm{id}$ of this neighborhood is a tubular
neighborhood of $D\times_M Z$ in $Z\times \bR^n$ which we identify
with the normal bundle $\nu(k)$.  There are canonical
identifications
\begin{align*}
  \nu(k) \cong (D\times_M Z) \times \bR^n, \\
  \nu((f,j)) \cong \nu(f) \oplus \bR^n,
\end{align*}
and the projection $q: D\times_M Z \to D$ identifies the trivial
factor of $\nu((f,j))$ with $\nu(k)$.  See Figure
\ref{transfer-pt-figure}.  Using these choices it is easy to check
that either composition in \eqref{transfer-pt-spectra-diagram} sends
$(z,x) \in Z\times \bR^n$ to $\infty$ if it is not in the tubular
neighborhood of $D\times_M Z$ in $Z\times \bR^n$; and if $(z,x)$
does lie in the tubular neighborhood of $D\times_M Z$, corresponding
to a point $((d,z),y) \in \nu(k)$, then it is sent to $((d,0),y) \in
\nu(f) \oplus \bR^n$. Letting $n \to \infty$ and taking the adjoint diagram completes the proof in
the case of manifolds. The statement for stacks follows from this,
because any homotopy class of maps $K \to \Ho(Z)$, $K$ a finite
CW complex, can be represented by a submersion $K \dash \to Z$
(where $K \dash$ is a manifold homotopy equivalent to $K$).
\begin{figure}\label{transfer-pt-figure}
\input{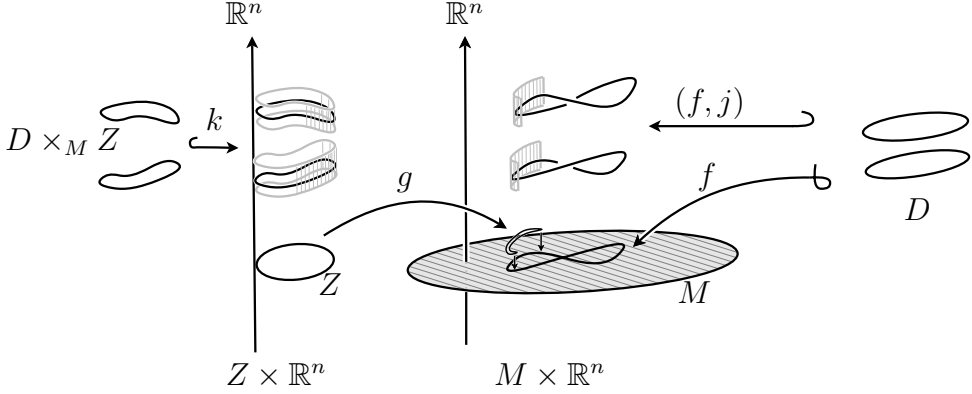}
\caption{Illustration of geometry of Lemma \ref{transfer-lemma}.  The
  normal bundle of $D$ canonically splits as $\nu(f) \oplus \bR^n$; the
  trivial factor and its pullback to $D\times_M Z$ are shown in grey.}
\end{figure}
\end{proof}

\begin{lemma}
The diagram \eqref{main-diagram} is weakly homotopy commutative.
\end{lemma}
\begin{proof}
 For each $\alpha \in A$ one sees that the morphisms
\[
\M((\Gamma_\alpha)) \looparrowright \mnbar{g}{n} \leftarrow
\fM((\Gamma))
\]
satisfy the hypotheses of Lemma \ref{transfer-lemma}.  One easily
verifies that
\[
\M((\Gamma_\alpha)) \times_{\mnbar{g}{n}} \fM((\Gamma)) \cong \fM((\Gamma | e)),
\]
where $e$ is an edge of the type specified by $\alpha$.  We denote the
projection onto the first factor by $p$.  Now consider the following
diagram:
\[
\xymatrix{
   B(G_\alpha \wr \Sigma_{m_\alpha})   \ar[d]^-{\trf}  
 & \fM((\Gamma)) \ar[l] \ar@/^1pc/[dr]^-{\xi_{E(\Gamma)}} \ar[d]_-{\trf} \\ 
   \Q B(G_\alpha \wr (\Sigma_{m_\alpha - 1} \times 1))_+ \ar[d]^-{\Q \proj} 
 & \Q \fM((\Gamma | e))_+ \ar[l] \ar[d]_-{\Q p} 
 & \mnbar{g}{n} \ar@/^1pc/[ddl]^-{\PT_{\xi_\alpha}} \\
   \Q (BG_\alpha)_+ \ar[d]^-{\Q \inc} 
 & \Q \M((\Gamma_\alpha))_+ \ar[d]_-{\Q \inc} \ar[l] \\ 
   \Q_m BG_\alpha^{V}  
 & \Q \M((\Gamma_\alpha))^{\nu(\alpha)}  \ar[l]}
\]
The right triangle is homotopy commutative by Lemma
\ref{transfer-lemma}.  The horizontal maps are all (induced by)
classifying maps of suitable vector bundles.  It is clear that the
left part of the diagram is also commutative and that the composition
along the left side is precisely the $\alpha$ component of the
counterclockwise composition in Lemma \ref{main-diagram}.
\end{proof}

\subsection{A quick review of homology of infinite loop
spaces}\label{infinite-loop-space-homology}

We recall the description of the homology of the free infinite loop
space $\Q X$ with coefficients in a field $\bF$.  In characteristic zero
the description is easy and classical; in finite characteristic the
standard reference is \cite{May}.

If $\bF$ is a field, $V$ a graded $\bF$-vector space, then we denote
by $\Lambda (V)$ the free graded-commutative $\bF$-algebra generated
by $V$.

Let $X$ be a pointed space.  There is a natural map $X \to \Q(X)$,
adjoint to the identity on $\suspinf X$ and thus a homomorphism
$H_*(X) \to H_*(\Q X)$.  Because $\Q X$ is a homotopy-commutative
$H$-space, the Pontrjagin product defines the structure of a
graded-commutative $\bF$-algebra on the homology $H_*(\Q X; \bF)$.
Thus we obtain a ring homomorphism $\Lambda (\widetilde{H}_*(X; \bF))
\to H_*(\Q X; \bF)$. If $\car(\bF)=0$ then this is an isomorphism.  This
is a standard result of algebraic topology, see \cite{MilMoo}, p. 262 f.

If $\car \bF = p > 0$ then the homology $H_*(\Q X; \bF)$ is much
richer.  The homology algebra $H_*(\Q X;\bF)$ is a module over an
algebra of homology operations known as the \emph{Dyer-Lashof
  operations} (they are also known as Araki-Kudo operations if $p=2$).
These operations measure the failure of chain-level commutativity of
the Pontrjagin product.

For $p\neq 2$, these operations are:
\[
\beta^\epsilon Q^s: H_n(\Q X;\bF) \to H_{n+2s(p-1) -\epsilon}(\Q
X;\bF)
\]
for $\epsilon \in \{0,1\}$ and $s\in \bZ_{\geq \epsilon}$.  Given a
sequence $I=(\epsilon_1, s_1, \ldots, \epsilon_n, s_n)$, $I$ is
\emph{admissible} if $s_{i+1} \leq p s_i - \epsilon_i$ for $i=1,
\ldots, n-1$.
One defines the \emph{excess}
\[
e(I) = 2s_1 -\epsilon_1 - \sum_{i=2}^n (2s_i(p-1)-\epsilon_i)
\]
and $b(I)=\epsilon_1$.  Such a sequence determines an iteration of
operations which is written $Q^I$.

When $p=2$ the operations are of the form
\[
Q^s: H_n(\Q X;\bF) \to H_{n+s}(\Q X;\bF)
\]
for $s \in \bZ_{\geq 0}$.  A sequence $I=(s_1, \ldots, s_n)$ is
\emph{admissible} if $s_{i+1}\leq 2s_i$ for each $i=1, \ldots, n-1$.
The excess is defined to be $e(I) = s_1 - \sum_{i=2}^n s_i$, and for
convenience one puts $b(I)=0$.

Let $V$ be a graded $\bF$-vector space and let $B$ be a homogeneous
basis of $V$. The \emph{free unstable Dyer-Lashof module generated
by $V$} is the $\bF$-vector space $\DL_{\bF} (V) $ on the basis
\[
\{Q^I x \:\: | \:\: x\in B, \mbox{ $I$ admissible}, e(I) + b(I) \geq
\mathrm{deg}(x) \}.
\]
Because $H_*(\Q X; \bF)$ has Dyer-Lashof operations, there is a ring
homomorphism, compatible with the Dyer-Lashof operations
\[\Lambda( \DL_{\bF} (\widetilde{H}_* (X; \bF))) \to H_*(\Q X; \bF),\]
and it is proven in \cite{May} that this is an isomorphism. This
calculation immediately implies Lemma \ref{dyerlashof}.

\section{Comparison to the tautological algebra}\label{comparisontautological}

Here we explain the relationship between the rational cohomology
classes detected via Theorem \ref{mainthm} and the tautological
algebra of $\mnbar{g}{n}$.

\begin{proposition}\label{tautological}
  The image of the homomorphism 
  \[
  \Phi_{irr}^{*}: H^{*} (\Q BN(2)^V; \bQ) \to H^{*} (\mnbar{g}{n};
  \bQ)
  \]
  is contained in the cohomology tautological algebra $\cR^*
  (\mnbar{g}{n})$. The analogous statement is true for the other maps
  studied in Theorem \ref{mainthmspecial} and Theorem \ref{mainthm}.
\end{proposition}

Before we can explain the definition of $\cR^*(\mnbar{g}{n})$ and the
proof of \ref{tautological}, we need to say a few words about umkehr
maps (also called ``pushforward'' or ``Gysin map'') in cohomology and
their relation to the Pontrjagin-Thom construction.

Let $f:M \to N$ be a proper smooth map between manifolds (or a proper
representable morphism between differentiable local quotient stacks)
of codimension $d$, and let $\PT_f:\suspinf N_+ \to \bTh (\nu(f))$ be
its Pontrjagin-Thom map. A \emph{cohomological orientation} of $f$ is
by definition a Thom class in $H^d (\bTh(\nu(f))$. This orientation
induces a \emph{Thom isomorphism} $\thom:H^* (M) \to H^{*+d}
(M^{\nu(f)})$ (see \cite{Rudyak}, ch. V for details). The \emph{umkehr
  map} $f_{!}$ is defined as the composition
\begin{align}\label{umkehr}
  H^{*}(M) \cong H^{*}(\suspinf M_+)
    \stackrel{\thom}{\longrightarrow} H^{*+d}(M^{\nu(f)}) \\
    \stackrel{\PT_{f}^{*}}{\longrightarrow}
  H^{*+d}(\suspinf N_+) \cong   H^{*+d}(N).
\end{align}

The tautological algebra has been studied by many authors; we refer to
the survey paper \cite{Vakil}. Here is the definition. One considers
all natural morphisms $\mnbar{g}{n+1} \to \mnbar{g}{n}$ (forget the
last point and collapse an unstable component if it shows up),
$\mnbar{g-1}{n+2} \to \mnbar{g}{n}$, $\mnbar{h}{k+1} \times
\mnbar{g-h}{n-k+1} \to \mnbar{g}{n}$ (the gluing morphisms) and $
\mnbar{g}{n} \to \mnbar{g}{n}$ (given by a permutation of the
labelling set $\{1, \ldots , n \}$).  All these morphisms are
representable morphisms of complex-analytic stacks and so they have
canonical orientations. Thus there are umkehr maps in integral
cohomology for these morphisms. There is another, more traditional way
to define the umkehr maps for complex orbifolds, based on rational
Poincar\'e duality for the coarse moduli spaces, but this only works
in rational cohomology.

\begin{definition}
  The collection of tautological algebras 
  \[
  \cR^* (\mnbar{g}{n}) \subset H^{2*}(\mnbar{g}{n}; \bQ)
  \]
  is the smallest system of unital $\bQ$-subalgebras which contain all
  classes $\psi_i = c_1 (L_i) \in H^2 (\mnbar{g}{n}; \bQ)$, for all
  $g$, $n$ and $i=1, \ldots ,n$ and which is closed under pushforward
  by the natural morphisms above.
\end{definition}

We prove Proposition \ref{tautological} only for the map $\Phi_{irr}:
\mnbar{g}{n} \to \Q BN(2)^V$, which is sufficient to clarify the
pattern.

First recall that $H^*(\Q BN(2)^V; \bQ) = \bQ [a_{i,j}]$, where
\[
a_{i,j} = \thom(y_{1}^{i}y_{2}^{j}) \in H^{2 + 2i + 4j} (BN(2)^V),\]
and $y_i$ is the $i$th Chern class of the $2$-dimensional complex
vector bundle on $BN(2)$ induced by the inclusion $N(2) \to U(2)$.
Thus we need to argue that $\Phi_{irr}^{*} (\thom (y_{1}^{i}
y_{2}^{j}))$ is in the tautological algebra. By the definition of
$\Phi_{irr}$, this is nothing else than $\PT_{\xi_{irr}}^{*} (\thom(
c_1(W)^{i} c_2(W)^{j}))$, where $W=L_{n+1} \oplus L_{n+2} \to
\mnbar{g-1}{n+(2)}$ is the sum of the natural line bundles (which is
well-defined, although the last two points are permuted). This can be
rewritten, using the definition of the umkehr map, as 
\[
(\xi_{irr})_{!} (c_1(W)^{i} c_2(W)^{j}) = (\xi_{irr})_{!}
((\psi_{n+1} + \psi_{n+2})^{i} (\psi_{n+1} \psi_{n+2})^{j}) .
\]
This obviously lies in the tautological ring. There is a little
argument needed, because we used the PT-map starting from
$\mnbar{g-1}{n+(2)}$, while the tautological algebra is defined using
the map from the $2$-fold cover $\mnbar{g-1}{n+2}$. We leave this to
the reader.

\bibliographystyle{amsalpha}
\bibliography{dm-bib}

\providecommand{\bysame}{\leavevmode\hbox to3em{\hrulefill}\thinspace}
\providecommand{\MR}{\relax\ifhmode\unskip\space\fi MR }
\providecommand{\MRhref}[2]{%
  \href{http://www.ams.org/mathscinet-getitem?mr=#1}{#2}
}
\providecommand{\href}[2]{#2}
\begin{thebibliography}{Noo05b}

\bibitem[Beh04]{Behr}
K.~Behrend, \emph{Cohomology of stacks}, Intersection theory and moduli, ICTP
  Lect. Notes, XIX, Abdus Salam Int. Cent. Theoret. Phys., Trieste, 2004,
  pp.~249--294.

\bibitem[Ber81]{Bers}
L.~Bers, \emph{Finite-dimensional {T}eichm\"uller spaces and generalizations},
  Bull. Amer. Math. Soc. (N.S.) \textbf{5} (1981), no.~2, 131--172.

\bibitem[Bet02]{Bet}
S.~Betley, \emph{Twisted homology of symmetric groups}, Proc. Amer. Math. Soc.
  \textbf{130} (2002), no.~12, 3439--3445 (electronic).

\bibitem[BT01]{BoeTill}
C.-F. B{\"o}digheimer and U.~Tillmann, \emph{Stripping and splitting decorated
  mapping class groups}, Cohomological methods in homotopy theory (Bellaterra,
  1998), Progr. Math., vol. 196, Birkh\"auser, Basel, 2001, pp.~47--57.

\bibitem[BtD95]{BroetDie}
T.~Br\"ocker and T.~tom Dieck, \emph{Representations of compact {L}ie groups},
  Springer-Verlag, New York, 1995.

\bibitem[DM69]{DM}
P.~Deligne and D.~Mumford, \emph{The irreducibility of the space of curves of
  given genus}, Inst. Hautes \'Etudes Sci. Publ. Math. \textbf{36} (1969),
  75--109.

\bibitem[Edi00]{Edi}
D.~Edidin, \emph{Notes on the construction of the moduli space of curves},
  Recent progress in intersection theory (Bologna, 1997), Trends Math.,
  Birkh\"auser Boston, Boston, MA, 2000, pp.~85--113.

\bibitem[FHT]{FHT}
D.~Freed, M.~Hopkins, and C.~Teleman, \emph{{L}oop groups and twisted
  {K}-theory {I}}, preprint, arXiv:0711.1906, 2007.

\bibitem[Gal04]{Galatius}
S.~Galatius, \emph{Mod {$p$} homology of the stable mapping class group},
  Topology \textbf{43} (2004), no.~5, 1105--1132.

\bibitem[GE06]{Galatius-Eliashberg}
S.~Galatius and Y.~Eliashberg, \emph{Homotopy theory of compactified moduli
  spaces}, Oberwolfach Reports (2006), 761--766.

\bibitem[GK98]{GK}
E.~Getzler and M.~M. Kapranov, \emph{Modular operads}, Compositio Math.
  \textbf{110} (1998), no.~1, 65--126.

\bibitem[Hae84]{Haef}
A.~Haefliger, \emph{Groupo\"\i des d'holonomie et classifiants}, Ast\'erisque
  (1984), no.~116, 70--97, Transversal structure of foliations (Toulouse,
  1982).

\bibitem[Han07]{Hanb}
L.~Hanbury, \emph{Homology stability of mapping class groups and open-closed
  cobordism categories}, Ph.D. thesis, Oxford, 2007.

\bibitem[Har85]{Har}
J.~L. Harer, \emph{Stability of the homology of the mapping class groups of
  orientable surfaces}, Ann. of Math. \textbf{121} (1985), no.~2, 215--249.

\bibitem[Hei05]{Hein}
J.~Heinloth, \emph{Somes notes on differentiable stacks}, Mathematisches
  Institut, Seminars 2004/05, Universit\"at G\"ottingen (2005), 1--32.

\bibitem[HM98]{HarMor}
J.~Harris and I.~Morrison, \emph{Moduli of curves}, Graduate Texts in
  Mathematics, vol. 187, Springer-Verlag, New York, 1998.

\bibitem[HW07]{HatWah}
A.~Hatcher and N.~Wahl, \emph{Stabilization for mapping class groups of
  3-manifolds}, preprint, Ar{X}iv:{math/0709.2173}, 2007.

\bibitem[Iva93]{Iv}
N.~V. Ivanov, \emph{On the homology stability for {T}eichm\"uller modular
  groups: closed surfaces and twisted coefficients}, Mapping class groups and
  moduli spaces of Riemann surfaces (G\"ottingen, 1991/Seattle, WA, 1991),
  Contemp. Math., vol. 150, Amer. Math. Soc., Providence, RI, 1993,
  pp.~149--194.

\bibitem[Knu83]{Knud}
F.~Knudsen, \emph{The projectivity of the moduli space of stable curves. {II}.
  {T}he stacks {$M\sb{g,n}$}}, Math. Scand. \textbf{52} (1983), no.~2,
  161--199.

\bibitem[KT76]{KanThu}
D.~M. Kan and W.~P. Thurston, \emph{Every connected space has the homology of a
  {$K(\pi ,1)$}}, Topology \textbf{15} (1976), no.~3, 253--258.

\bibitem[LMB00]{Laumon}
G.~Laumon and L.~Moret-Bailly, \emph{Champs alg\'ebriques}, Ergebnisse der
  Mathematik und ihrer Grenzgebiete. 3. Folge. A Series of Modern Surveys in
  Mathematics, vol.~39, Springer-Verlag, Berlin, 2000.

\bibitem[May76]{May}
J.~P. May, \emph{The homology of ${E}_\infty$-spaces}, The homology of iterated
  loop spaces, Springer-Verlag, Berlin, 1976, Lecture Notes in Mathematics,
  Vol. 533.

\bibitem[MM65]{MilMoo}
J.~W. Milnor and J.~C. Moore, \emph{On the structure of {H}opf algebras}, Ann.
  of Math. (2) \textbf{81} (1965), 211--264.

\bibitem[Moe02]{Moer}
I.~Moerdijk, \emph{Orbifolds as groupoids: an introduction}, Orbifolds in
  mathematics and physics (Madison, WI, 2001), Contemp. Math., vol. 310, Amer.
  Math. Soc., Providence, RI, 2002, pp.~205--222.

\bibitem[Mos57]{most}
G.~D. Mostow, \emph{Equivariant embeddings in {E}uclidean space}, Ann. of Math.
  \textbf{65} (1957), 432--446.

\bibitem[MS76]{McDuff-Segal}
D.~McDuff and G.~Segal, \emph{Homology fibrations and the ``group-completion''
  theorem}, Invent. Math. \textbf{31} (1975/76), no.~3, 279--284.

\bibitem[MW07]{Madsen-Weiss}
I.~Madsen and M.~Weiss, \emph{The stable moduli space of {R}iemann surfaces:
  {M}umford's conjecture}, Ann. of Math. \textbf{165} (2007), 843--941.

\bibitem[Noo05a]{Noo2}
B.~Noohi, \emph{Cohomology of pretopological stacks}, preprint, available at
  {http://www.math.fsu.edu/$\sim$noohi/papers/papers.html}, 2005.

\bibitem[Noo05b]{Noo1}
\bysame, \emph{Foundations of topological stacks {I}}, preprint,
  Ar{X}iv:{math/0503247}, 2005.

\bibitem[RS06]{SalRob}
J.~W. Robbin and D.~A. Salamon, \emph{A construction of the {D}eligne-{M}umford
  orbifold}, J. Eur. Math. Soc. (JEMS) \textbf{8} (2006), no.~4, 611--699.

\bibitem[Rud98]{Rudyak}
Y.~B. Rudyak, \emph{{O}n {T}hom spectra, orientability, and cobordism},
  Springer-Verlag, 1998.

\bibitem[Seg68]{segclass}
G.~Segal, \emph{Classifying spaces and spectral sequences}, Inst. Hautes
  {\'E}tudes Sci. Publ. Math. \textbf{34} (1968), 105--112.

\bibitem[Seg73]{Segconf}
\bysame, \emph{Configuration-spaces and iterated loop-spaces}, Invent. Math.
  \textbf{21} (1973), 213--221.

\bibitem[Seg74]{Segcat}
\bysame, \emph{Categories and cohomology theories}, Topology \textbf{13}
  (1974), 293--312.

\bibitem[Ste67]{steen}
N.~E. Steenrod, \emph{A convenient category of topological spaces}, Michigan
  Math. J. \textbf{14} (1967), 133--152.

\bibitem[Vak06]{Vakil}
R.~Vakil, \emph{The moduli space of curves and {G}romov-{W}itten theory},
  preprint, ArXiv:math.AG/0602347, 2006.

\bibitem[Zun06]{Zung}
N.~T. Zung, \emph{Proper groupoids and momentum maps: linearization, affinity,
  and convexity}, Ann. Sci. {\'E}cole Norm. Sup. (4) \textbf{39} (2006),
  841--869.

\end{thebibliography}

\end{document}